\documentclass{amsart}
\usepackage{amsmath, amsthm, amssymb}
\usepackage[top=1in, bottom=1in, right=1in, left=1in]{geometry}
\usepackage{mathtools}
\usepackage{mathrsfs}
\usepackage{xspace}
\usepackage[utf8]{inputenc}
\usepackage{hyperref}
\usepackage{graphicx}
\usepackage{xcolor}
\usepackage{textcomp}
\usepackage[shortlabels]{enumitem}

\usepackage{tikz}
\usetikzlibrary{arrows}
\usetikzlibrary{shapes.geometric}
\usetikzlibrary{matrix}
\usetikzlibrary{decorations}

\setlist{itemsep=0em}
\tikzstyle{vertex}=[rectangle,
  draw,
  fill=white,
  opacity=1,
  text opacity=1,
  inner sep=3pt,
  thick,
  align=center,
minimum size=1.6em]

\tikzstyle{a}=[-angle 90, draw=color0, thick]
\tikzstyle{b}=[-angle 90, draw=color1, thick]
\tikzstyle{c}=[-angle 90, draw=color2, thick]

\pgfdeclaredecoration{sl}{initial}{
  \state{initial}[width=\pgfdecoratedpathlength-1sp]{
    \pgfmoveto{\pgfpointorigin}
  }
  \state{final}{
    \pgflineto{\pgfpointorigin}
  }
}
\pgfarrowsdeclarecombine{aa}{aa}{angle 90}{angle 90}{angle 90}{angle 90}

\tikzstyle{solid black}=[->, draw=black, thick, -angle 90]
\tikzstyle{dashed black}=[->, draw=black, thick, dashed, -angle 90]

\tikzset{arrow/.style={-angle 90, shorten >=2pt, shorten <=2pt}}
\tikzset{parallel_arrow_1/.style={-angle 90,
decoration={sl,raise=1mm},decorate, shorten >=2pt, shorten <=2pt}}
\tikzset{parallel_arrow_2/.style={-angle 90,
decoration={sl,raise=-1mm},decorate, shorten >=2pt, shorten <=2pt}}
\tikzset{aarrow/.style={-aa, shorten >=2pt, shorten <=2pt, densely dashed}}
\tikzset{parallel_aarrow_1/.style={-aa,
    decoration={sl,raise=1mm},decorate, shorten >=2pt, shorten <=2pt,
densely dashed}}
\tikzset{parallel_aarrow_2/.style={-aa,
    decoration={sl,raise=-1mm},decorate, shorten >=2pt, shorten <=2pt,
densely dashed}}

\usepackage[capitalise, noabbrev, nameinlink]{cleveref}

\hypersetup{
  colorlinks,
  linkcolor={red!50!black},
  citecolor={blue!50!black},
  urlcolor={blue!80!black}
}

\newtheorem{theorem}{Theorem}[section]
\newtheorem{lem}[theorem]{Lemma}
\newtheorem{prop}[theorem]{Proposition}
\newtheorem{cor}[theorem]{Corollary}

\theoremstyle{definition}

\newtheorem{ex}[theorem]{Example}

\numberwithin{theorem}{section}
\numberwithin{theorem}{section}

\DeclareMathOperator{\im}{im}
\DeclareMathOperator{\dom}{dom}
\DeclareMathOperator{\Stab}{Stab}
\DeclareMathOperator{\Ker}{Ker}
\DeclareMathOperator{\Tr}{Tr}
\DeclareMathOperator{\rank}{rank}

\newcommand{\N}{\mathbb{N}}
\newcommand{\J}{\mathscr{D}}
\newcommand{\R}{\mathscr{R}}
\renewcommand{\L}{\mathscr{L}}
\newcommand{\D}{\mathscr{D}}
\renewcommand{\H}{\mathscr{H}}
\newcommand{\genset}[1]{\langle #1 \rangle}
\newcommand{\set}[2]{\{ #1 \mid #2\}}
\newcommand{\defn}[1]{\textbf{\textit{#1}}}

\newcommand{\GAP}{\textsc{GAP}~\cite{GAP4}\xspace}

\renewcommand{\to}{\longrightarrow}

\usepackage[backend=biber, bibencoding=utf8, giveninits=true,
maxbibnames=99]{biblatex}
\title{Computing congruences of finite inverse semigroups}
\author{Luna Elliott, Alex Levine and James D. Mitchell}
\subjclass[2020]{20M18, 20-08, 20M30}
\keywords{inverse semigroup, congruence, computational algebra}
\begin{document}
\begin{abstract}
  In this paper we present a novel algorithm for computing a congruence on an
  inverse semigroup from a collection of generating pairs. This algorithm uses
  a myriad of techniques from the theories of groups, automata,
  and inverse semigroups. An initial implementation of this algorithm
  outperforms existing implementations by several orders of magnitude.
\end{abstract}

\maketitle

\section{Introduction}

In this paper we are concerned with the question of computing two-sided
congruences of a finite inverse semigroup. The class of inverse semigroups lies
somewhere between the classes of groups and semigroups, with more useful
structure than semigroups in general, but less structure than groups.
A \defn{semigroup} is a set, usually denoted $S$, with an associative binary
operation, usually indicated by juxtaposing elements of $S$. An \defn{inverse
semigroup} is a semigroup $S$ such that for every element $s\in S$ there exists
a unique $s'\in S$ with $ss's=s$ and $s'ss'=s'$. The element $s'$ is usually
denoted $s^{-1}$, a choice which is at least partially justified by the fact
that if $S$ is a group and $s\in S$, then $s^{-1}$ is just the usual group
theoretic inverse of $s$. On the other hand, if $S$ is not a group and $s\in
S$, then neither $ss^{-1}$ nor $s^{-1}s$ is necessarily equal to the identity
of $S$, not least because $S$ need not have an identity.

Two-sided congruences are to semigroups what normal subgroups are to groups.
However, a \defn{two-sided congruence} on a semigroup $S$ is not a subsemigroup
of $S$ but rather an equivalence relation $\rho \subseteq S\times S$ with the
property that if $(x, y)\in \rho$ and $s\in S$, then $(xs, ys), (sx, sy)\in
\rho$. Although there is a definition of one-sided congruences also, akin to the
notion of subgroups of groups, we will be solely concerned with two-sided
congruences, and so we will drop the ``two-sided'' and henceforth refer to
``congruences'' to exclusively mean ``two-sided congruences''. Equivalently,
$\rho$ is a congruence on $S$ if $(xs, yt)\in \rho$ whenever $(x, y), (s, t)\in
\rho$, making $\rho$ a subsemigroup of $S\times S$ rather than $S$. Although
congruences of semigroups and normal subgroups of groups are analogous notions,
the definition for semigroups is a special case of that for
universal algebras; see, for example,~\cite[Section 5]{Burris2011}. For a
congruence \(\rho\) of a semigroup $S$, it will often, but not always, be
convenient for us to write \(x=_\rho y\) instead of \((x,y)\in \rho\).

If $G$ is a group and $A\subseteq G$, then algorithms for determining the least
normal subgroup $\genset{\genset{A}}$ of $G$ containing $A$ are one of the core
components of computational group theory. Following the nomenclature of \GAP we
refer to such algorithms as \defn{normal closure} algorithms. For example,
normal closure algorithms for permutation groups are considered in \cite[Section
5.4.1]{Seress2003}; for groups in general, in \cite[Section 3.3.2]{Holt2005};
or for computing all normal subgroups in~\cite{Hulpke1998}.


Congruences of inverse semigroups have also been studied extensively in the
literature. For example, the lattices of congruences of various semigroups,
including many inverse semigroups, have been completely described, for example
in \cite{Liber, Malcev53,Malcev52, Wang}; and from the perspective of
computation in \cite{Merkouri2023aa, Araujo2022aa, Coleman2024, Torpey}. With
the single exception of~\cite{Torpey}, the existing
algorithms~\cite{Merkouri2023aa, Araujo2022aa, Coleman2024}, and their
implementations in ~\cite{libsemigroups,Semigroups, CREAM}, for computing
individual congruences on an inverse semigroup do not use any of the specific
structure of inverse semigroups. We will say more about the exception below.

The following notions have been, and will be here, indispensable for the study
of congruences on inverse semigroups.
Let $S$ be an inverse semigroup. We denote the set of idempotents of $S$ by
$E(S)$; and note that $E(S)$ is an inverse subsemigroup of $S$.
The \defn{kernel} of a congruence $\rho$ on an inverse semigroup $S$
is the inverse subsemigroup
\[
  \Ker(\rho) = \set{s\in S}{\text{there exists }e\in E(S),\ s=_\rho e}\leq S
\]
and the \defn{trace} of $\rho$ is the restriction of $\rho$ to the
idempotents of $S$:
\[
  \Tr(\rho) = \rho \cap \left(E(S)\times E(S)\right).
\]

If $S$ is an inverse semigroup and $\rho$ is a congruence on $E(S)$, then $\rho$
is said to be \defn{normal in $S$} if $s^{-1}xs =_\rho s^{-1}ys$ whenever $s\in
S$ and $x=_\rho y$. Similarly, if $T$ is an inverse subsemigroup of
$S$, then $T$ is
\defn{normal in $S$} if $s^{-1}ts\in T$ for all $s\in S$ and all $t\in T$. If
$T$ is a normal inverse subsemigroup of $S$ and $\tau$ is a normal
congruence on $E(S)$,
then $(T, \tau)$ is (rather unimaginatively) called a \defn{congruence pair} if
the following conditions hold:
\begin{enumerate}[label=(CP\arabic*)]
  \item\label{item-CP1} $ae\in T$ and $e=_\tau a^{-1}a$ implies that $a\in T$;
  \item\label{item-CP2} $a\in T$ implies that $aa^{-1} =_\tau a^{-1}a$;
\end{enumerate}
for all $a\in S$ and all $e\in E(S)$. It is well-known that the congruences of
an inverse semigroup $S$ are in one-to-one correspondence with the congruence
pairs on $S$; see, for example,~\cite[Theorem 5.3.3]{Howie}. The kernel-trace
description originates in \cite{Scheiblich}, and is described in almost all
books about semigroup theory; in addition to~\cite[Theorem 5.3.3]{Howie},
see~\cite[Proposition 1.3]{Grillet2017},~\cite[Section
5.1]{Lawson1998},or~\cite[Chapter III]{petrich_book}.

In this paper, we present various mathematical results that can be combined into
an algorithm for computing a congruence on a finite inverse semigroup. The aim
of these results is to allow the efficient computation of the least congruence
$R^{\sharp}$ on an inverse semigroup $S$ from a collection of generating pairs
$R \subseteq S\times S$. By ``compute'' a congruence $\rho$ we mean that we have
a representation of the congruence that is amenable to computation (i.e. that is
not larger than necessary, and can be computed relatively quickly), and that can
be used to answer questions about $\rho$ such as whether or not $(x,y)\in
S\times S$ belongs to \(\rho\); what is the number of classes in $\rho$; and
what are the elements of $x/\rho$?

A preliminary implementation of these algorithms in \cite{GAP4} and
\cite{libsemigroups}, indicates that there is, for some examples,
at least a quadratic speedup in comparison to the existing implementation
of~\cite{Torpey} in \cite{Semigroups} (which uses the kernel and trace); and
the implementation in \cite{libsemigroups} and \cite{Semigroups}
(for semigroups in general); see~\cref{appendix-perf} for details.

Although significantly faster than existing implementations, it is worth
mentioning that neither the time nor space complexity of the algorithms we
present is polynomial in the size of the input. For instance, one key step in
the algorithm we present is computing the trace of a congruence on an inverse
semigroup $S$. If $S$ is the symmetric inverse monoid on the set $\{1, \ldots,
n\}$, then $S$ can be represented using $O(n)$ space. However, $|E(S)|=2^n$, and
computing the idempotents in this case has complexity $O(2 ^n)$. The complexity
of the other steps in the algorithm are somewhat harder to describe; but they
also depend on $|E(S)|$. It seems unlikely to the authors that there is a
sub-exponential algorithm for computing a congruence on an inverse semigroup.
Again we refer the reader to~\cref{appendix-perf} for a more detailed
discussion.

The paper is organized as follows. In \cref{section-preliminaries} we provide
some details of the prerequisite notions from semigroup theory that we require.
In \cref{section-data-structure} we describe data structures for inverse
semigroups, and their quotients, that uses the theory of Green's relations, the
action of an inverse semigroup on its idempotents by conjugation, and an
analogue of Schreier's Lemma. The data structure consists of a generating set
$X$ for the inverse semigroup $S$, a certain automata-like graph $\Gamma_X$
encoding the action (of the previous sentence) and its strongly connected
components, and a finite sequence $G_1, \ldots, G_m$ of groups. For a quotient
of $S$, the data structure consists of the generating set $X$ for $S$, a
quotient of the graph $\Gamma_X$, and a sequence of normal subgroups $N_1,
\ldots, N_n$ of the groups in the data structure for $S$. In
\cref{section-compute-the-trace} we describe how to compute the trace of a
congruence using $\Gamma_X$ and a (guaranteed to terminate) variant of the
Todd-Coxeter Algorithm from~\cite{Coleman2024}. In
\cref{section-compute-the-groups}, we show how to obtain relatively small
collections of elements $Y_i$ of each group $G_i$ such that the normal closure
$\genset{\genset{Y_i}}$ is the required normal subgroup $N_i$. In
\cref{section-compute-the-kernel} we show how to obtain the elements of an
arbitrary class of a congruence, and apply this to determine the elements of
the kernel as a translate of the preimage of a coset of a normal subgroup under
a homomorphism of groups. In \cref{section-membership}, we discuss how to test
whether or not a pair of elements of an inverse semigroup belong to a
congruence. In \cref{section-meets-joins}, we indicate how to use the
Hopcroft-Karp Algorithm~\cite{Hopcroft1971aa} and a standard algorithm from
automata theory, for finding a finite state automata recognising the
intersection of two languages, to compute joins and meets of congruences on
inverse semigroups represented by the data structure described in
\cref{section-data-structure}. In the final section,
\cref{section-max-idempotent-sep}, we describe a completely separate algorithm
for computing the maximum idempotent separating congruence on an inverse
subsemigroup of a finite symmetric inverse monoid.

\section{Preliminaries}\label{section-preliminaries}

Let $S$ be an inverse semigroup. We denote the set of idempotents of $S$ by
$E(S)$. If $s, t \in S$, then we write $s \leq t$ if there exists $e\in E(S)$
such that $s = te$. The relation $\leq$ is a partial order on $S$ (see, for
example, \cite[Proposition 1.4.7]{Lawson1998}), usually referred to as the
\defn{natural partial order} on $S$. This definition may appear to be
inherently ``right-handed'', but it is not, since $s \leq t$ if and only if
there exists $f\in E(S)$ such that $s = ft$ \cite[Lemma 1.4.6]{Lawson1998}.
Similarly, if $s \leq t$ and $u \leq v$, then $su \leq tv$ \cite[Proposition
1.4.7]{Lawson1998} and $xey \leq xy$ for all $x, y\in S$ and $e\in E(S)$.

We define a \defn{word graph} $\Gamma = (N, E)$ over the alphabet $A$ to be
a directed graph with nodes $N$ and edges $E\subseteq N \times A \times N$.
Word graphs are just finite state automata without initial or terminal
states.

If $(\alpha, a, \beta)\in E$ is an edge in a word graph $\Gamma = (N, E)$, then
$\alpha$ is the \defn{source}, $a$ is the \defn{label}, and $\beta$ is the
\defn{target} of $(\alpha, a, \beta)$. A word graph $\Gamma$ is \defn{complete}
if for every node $\alpha$ and every letter $a\in A$ there is at least one edge
with source $\alpha$ labelled by $a$. A word graph $\Gamma = (N, E)$ is
\defn{finite} if the sets of nodes $N$ and edges $E$ are finite. A word graph
is \defn{deterministic} if for every node $\alpha\in N$ and every $a\in A$
there is at most one edge with source $\alpha$ and label $a$. Complete
deterministic word graphs are just  unary algebras with universe $N$ and
operations $f_a: N \to N$ defined by $(\alpha)f_a = \beta$ whenever $(\alpha,
a, \beta)$ is an edge in $\Gamma$; see~\cite{Burris2011} for more details.  The
perspective of unary algebras maybe helpful, for those familiar with this
notion, when we define word graph quotients and homomorphisms, for complete
word graphs these are
identical to the notions of quotients and homomorphisms of the associated unary
algebras. If $\alpha, \beta \in N$, then an \defn{$(\alpha, \beta)$-path} is a
sequence of edges $(\alpha_0, a_0, \alpha_{1}), \ldots, (\alpha_{n - 1}, a_{n -
1}, \alpha_{n})\in E$ where $\alpha_0 = \alpha$ and $\alpha_{n} = \beta$ and
$a_0, \ldots, a_{n - 1}\in A$. If $\alpha, \beta \in V$ and there is an
$(\alpha, \beta)$-path in $\Gamma$, then we say that $\beta$ is
\defn{reachable} from $\alpha$. If $\alpha$ is a node in a word graph
$\Gamma$, then the \defn{strongly connected component of $\alpha$} is the set
of all nodes $\beta$ such that $\beta$ is reachable from $\alpha$ and
$\alpha$ is reachable from $\beta$. If $\Gamma_1= (N_1, E_1)$ and $\Gamma_2=
(N_2, E_2)$ are word graphs over the same alphabet $A$, then $\phi:N_1\to
N_2$ is a \defn{homomorphism} if $(\alpha, a, \beta)\in E_1$ implies
$((\alpha)\phi, a, (\beta)\phi)\in E_2$. If $\kappa$ is an equivalence
relation on the nodes of a word graph $\Gamma = (N, E)$, then we define the
\defn{quotient $\Gamma/\kappa$ of  $\Gamma$ by $\kappa$} to be the word graph
with nodes $\set{\alpha/\kappa}{\alpha \in N}$ and edges $\set{(\alpha
/\kappa, a, \beta/\kappa)}{(\alpha, a, \beta) \in E}$. Of course, even if
$\Gamma$ is deterministic, the quotient $\Gamma/\kappa$ is not necessarily
deterministic. If $\Gamma$ is deterministic, then $\Gamma/\kappa$ is
deterministic if and only if $\kappa$ is a congruence on the unary algebra
associated to $\Gamma$.

If $S$ is a semigroup, then we denote by $S^1$ either: $S\cup \{1_S\}$ with an
identity $1_S\not\in S$ adjoined; or just $S$ in the case that $S$ already has
an identity.

The final ingredient that we require in this paper is that of Green's relations.
If $s, t\in S$, then \defn{Green's $\mathscr{R}$-relation} is the equivalence
relation on $S$ defined by $(s, t)\in \R$ if and only if $sS ^{1} =
\set{sx}{x\in S^1} = tS^{1}$. Green's $\L$-relation is defined analogously;
Green's $\H$-relation is just $\L\cap \R$; and Green's $\D$-relations is defined
to be $\L\circ \R$. If $S$ is finite, then $(s, t)\in \D$ if and only if
$S^1sS^1 = S^1tS^1$.
A \defn{group $\H$-class} is an $\H$-class containing an idempotent, since it
forms a group under the same multiplication as $S$.
Green's relations are fundamental to the study of
semigroups; we refer the reader to any of \cites{Howie, Lawson1998, Grillet2017,
petrich_book} for further details. If $T$ is a subsemigroup of $S$ (denoted
$T\leq S$), then we may write $\mathscr{K}^S$ and $\mathscr{K}^T$ to distinguish
the Green's relations on $S$ and $T$ when $\mathscr{K}\in \{\L, \R, \H, \D\}$.
If $s\in S$, then we denote the equivalence class of Green's
$\mathscr{K}$-relation containing $s$ by $K_s$  or $K_s^S$ if we want to
indicate the semigroup containing the class.

\begin{theorem}[\textbf{Location Theorem}, cf. Proposition 2.3.7 in
  \cite{Howie}]\label{thm-product-location}
  Let $S$ be a finite semigroup and let $a, b\in S$ be such that $(a, b)\in \D$.
  Then the following are equivalent:
  \begin{enumerate}[\rm (a)]
    \item
      $(ab, a), (ab, b)\in \D$;
    \item
      $(a, ab)\in \R$ and $(ab, b)\in \L$;
    \item
      there exists an idempotent $e\in S$ such that $(e, a)\in \L$
      and $(e, b)\in \R$.
  \end{enumerate}
\end{theorem}

We will make repeated use of the following straightforward result also.

\begin{lem}\label{lem-D-leq}
  If \(S\) is a finite inverse semigroup, $e, f\in E(S)$ are such that $e\leq
  f$, and \((e, f) \in \mathscr{D}^S\), then \(e = f\).
\end{lem}

\section{A data structure for inverse semigroups and their
quotients}\label{section-data-structure}

In this section, we describe the data structure for inverse semigroups given
in~\cite[Section 5.6]{East2019aa}. We suppose that such an inverse semigroup
$S$ is given by a set of generators $X$ consisting of elements where both
products and equality of elements can be effectively computed. For example, $X$
may consist of functions from a finite set to itself (called
  \defn{transformations} in the semigroup literature, injective
  functions between subsets of a finite set (called \defn{partial
  permutations}), or matrices over a semiring. On the other hand, we do not
  consider the case, for example, where $S$ is defined by means of a
  presentation (consisting of generators and relations), since the
  problem of determining whether or not two words in
  the generators are equal is undecidable in general. We also intend that the
  data structure be used to represent a finite inverse semigroup, although the
  definition makes sense for infinite inverse semigroups too.

  The \defn{symmetric inverse monoid} $I_n$ for some $n\in \N$ is the set of all
  partial permutations of $\{1, \ldots, n\}$ with the operation of
  composition of binary relations. It might also be worth noting that, by the
  Vagner-Preston Representation Theorem (\cite[Theorem
  5.1.7]{Howie}), every inverse semigroup is isomorphic
  to a subsemigroup of some symmetric inverse monoid. As such from a
  mathematical perspective nothing would be lost by supposing that $S$ was an
  inverse subsemigroup of a symmetric inverse monoid. However, since we are
  concerned with practical computation, finding an inverse subsemigroup of a
  symmetric inverse semigroup that is isomorphic to $S$ may be prohibitively
  expensive, and since it is also not required we define our data structure
  without these assumptions and restrictions.

  If \(S\) is such an inverse semigroup, then the
  data structure for \(S\) consists of the following:
  \begin{enumerate}[label=(I\arabic*)]
    \item\label{item-A1} a generating set $X$ for \(S\);
    \item\label{item-A2} the word graph $\Gamma_{X}$ with nodes $E(S)$
      and edges $\{(e, x,
      x^{-1}ex)\mid e\in E(S), x\in X\}$;
    \item\label{item-A3} the strongly connected components of $\Gamma_X$;
    \item\label{item-A4} a generating set for the group $\mathscr{H}$-class
      $H_e$ of one representative $e\in E(S)$ in every strongly connected
      component of $\Gamma_X$.
  \end{enumerate}
  In the case that $S$ is finite, the word graph $\Gamma_X$ can be found in
  $O(|E(S)||X|)$ time and space (assuming that products in $S$ can be found in
  constant time). The strongly connected components of $\Gamma_X$ can be found
  from $\Gamma_X$ using algorithms from graph theory (such as those of
  Gabow~\cite{Gabow2000aa} or Tarjan~\cite{Tarjan1972aa}). Given the strongly
  connected components of $\Gamma_X$, the groups from \ref{item-A4} can be
  determined using the analogue of Schreier's Lemma given in \cite[Proposition
    2.3(c) and Algorithm
  3]{East2019aa}. For further context, the strongly connected
  components of $\Gamma_X$ are in
  1-1 correspondence with the $\D$-classes of $S$, and within a
  $\D$-class the group $\H$-classes are isomorphic as groups.
  Thus knowing a single group $\H$-class per $\D$-class means we know
  every group $\H$-class in the $\D$-class.
  This data structure can be used to answer many of the
  fundamental questions about $S$ that arise in a computational setting, such
  as membership testing in $S$, determining the Green's structure, and the size
  of $S$; see~\cite{East2019aa} for more details.

  If $S$ is an inverse semigroup $S$, $R\subseteq S \times S$, and $R
  ^{\sharp} = \rho$,
  we will show how to compute a data structure for the quotient $S/\rho$
  from the data structure for $S$. This data structure consists of:
  \begin{enumerate}[label=(Q\arabic*)]
    \item\label{item-Q1} the generating set $X$ for \(S\);
    \item\label{item-Q2}  the quotient word graph
      $\Gamma_{X}/\Tr(\rho)$ with nodes
      $E(S)/\Tr(\rho)$ and edges $\{(e/\Tr(\rho), x, (x^{-1}ex)/\Tr(\rho))\mid
      e\in E(S), x\in X\}$;
    \item\label{item-Q3} the strongly connected components of
      $\Gamma_X/ \Tr(\rho)$;
    \item\label{item-Q4} the generating sets for one group
      $\mathscr{H}$-class per strongly
      connected component of $\Gamma_X/\Tr(\rho)$.
  \end{enumerate}
  Clearly for \ref{item-Q2} we must compute $\Tr(\rho)$; and given
  \ref{item-Q2} we can compute the strongly connected components as we did for
  $S$ itself. Without a representation of $\rho$ (beyond $R$) we have
  no means of
  representing $X/\rho$, and hence we cannot determine the generating
  sets for the
  group $\mathscr{H}$-classes required in \ref{item-Q4}.
  We show how to compute $\Tr(\rho)$ from $R$ in
  \cref{section-compute-the-trace};
  and show how to compute the required group $\mathscr{H}$-classes
  in~\cref{section-compute-the-groups}.

  The quotient data structure is sufficient for representing the inverse
  semigroup $S/\rho$, and can be used to compute various aspects of $\rho$, such
  as the number of classes, or representatives of every class. But it does not
  suffice for other purposes, such as: computing the $\Ker(\rho)$, or, more
  generally, the elements of a congruence class $s/\rho$; or checking membership
  in $\rho$.  We describe one way of computing the kernel in
  \cref{section-compute-the-kernel} by providing an algorithm for finding the
  elements of a congruence class $s/\rho$ for a given $s\in S$. Checking
  membership in $\rho$ requires a means of testing membership in
  $\Ker(\rho)$. In \cref{section-membership} we show that the problem of testing
  membership in $\Ker(\rho)$ reduces to the problem of check membership in a
  coset of a normal subgroup of a group.

  It might be worth noting that none of the algorithms presented in this paper
  require any computation or representation of $\Ker(\rho)$ except the algorithm
  for computing $\Ker(\rho)$ itself.

  Throughout this paper we will use the notation from this section for $S$, the
  congruence $\rho$, and the associated data structures.

  \section{Computing the trace}\label{section-compute-the-trace}

  In this section we show how to compute the trace of a congruence on the
  inverse semigroup $S$ from the set of generating pairs $R\subseteq S \times
  S$.

  \begin{lem}[\textbf{Generating pairs for the
    trace}]\label{lem-gen-pairs-trace}
    If $S$ is an inverse semigroup and $\rho = R ^ {\sharp}$ is a congruence of
    $S$, then the trace $\Tr(\rho)$ of $\rho$ is the least normal congruence of
    $E(S)$ in $S$ containing
    \[
      \{(aea^{-1}, beb^{-1})\mid e\in E(S),\ (a, b)\in R\}.
    \]
  \end{lem}
  \begin{proof}
    Let $N$ denote the set of pairs in the statement, and let $\nu$ be the least
    normal congruence on $E(S)$ containing $N$. We must show that
    $\nu=\Tr(\rho)$.

    If $(a, b)\in R$ and $e\in E(S)$ are arbitrary, then certainly $a
    =_\rho b$ and
    so $ae =_\rho  be$ and $a^{-1} =_\rho  b^{-1}$. Hence $aea^{-1} =_\rho
    beb^{-1}$ and so $aea^{-1} =_{\Tr(\rho)}  beb ^{-1}$. Therefore
    $\nu\subseteq
    \Tr(\rho)$.

    For the converse containment, suppose that $e=_{\Tr(\rho)} f$. Then $e=_\rho
    f$, and hence $e=_{R ^ {\sharp}}
    f$. So there exist $s_0 = e, s_1, \ldots, s_n = f$
    where $s_i = p_iu_iq_i$ and $s_{i +1} = p_iv_iq_i$ for some $p_i,
    q_i \in S^1$
    and $(u_i, v_i)\in R$ for all $i$. We set $e_i = s_is_i^{-1}$ for every $i$.
    Then $e_0 = e$ and $e_n= f$. For every $i$, $e_i = s_is_i ^{-1} =
    p_iu_iq_iq_i^{-1}u_i ^{-1}p_i^{-1}$ and $e_{i + 1} = s_{i + 1} s_{i
    + 1} ^{-1} =
    p_iv_iq_iq_i^{-1}v_i ^{-1}p_i^{-1}$. Since $q_iq_i ^{-1} \in
    E(S)$ and $(u_i,
    v_i)\in R$, it follows that $(u_iq_iq_i^{-1}u_i ^{-1},
    v_iq_iq_i^{-1}v_i ^{-1})
    \in N\subseteq \nu$ by definition. Hence, since $\nu$ is normal, $e_i=
    p_iu_iq_iq_i^{-1}u_i ^{-1}p_i^{-1}=_\nu   p_iv_iq_iq_i^{-1}v_i
    ^{-1}p_i^{-1}=e_{i+1} $ for all $i$. Thus $e=e_0=_\nu e_n= f$, as required.
  \end{proof}

  For the remainder of this section we require \(S\) to be a monoid,
  by adjoining
  an identity $1_S$ if necessary. If $\sigma$ is any equivalence relation on
  $E(S)$, then we define $\Gamma_X/\sigma$ to be the word graph with nodes
  $E(S)/\sigma$ and edges $(e/\sigma, x, (x^{-1}ex)/\sigma)$ for all $e\in E(S)$
  and all $x\in X$. It is routine to verify that $\sigma$ is a normal congruence
  on $E(S)$ with respect to $S$ if and only if $\Gamma_X/\sigma$ is
  deterministic.
  In this case, $\sigma$ is completely determined by $\Gamma_X/\sigma$ as
  follows.

  \begin{lem}\label{lem-test-mem-trace}
    If $\sigma$ is any normal congruence on $E(S)$ and $e, f\in E(S)$, then
    $e=_{\sigma}f$ if and only if for all $x_1, \ldots, x_n, y_1, \ldots,
    y_m\in X$ such that $e = x_1\cdots x_n$ and $f = y_1\cdots y_m$ the words
    $x_1\cdots x_n$ and $y_1\cdots y_m$ both label $(1_S, e/\sigma)$-paths in
    $\Gamma_X/\sigma$.
  \end{lem}
  \begin{proof}
    If $e=x_1\cdots x_n\in E(S)$ where $x_i\in X$ labels a path from
    $1_S/\sigma$
    to $f/\sigma$ in $\Gamma_X/\sigma$, then $e=_{\sigma} f$.

    Conversely, if $e=_{\sigma} f$ and $e=x_1\cdots x_n$ and $f = y_1\cdots
    y_m$ where $x_i, y_j\in X$, then $x_1\cdots x_n$ and $y_1\cdots y_m$ both
    label $(1_S, e/\sigma)$-paths in $\Gamma_X/\sigma$.
  \end{proof}

  The next result is an immediate corollary of \cref{lem-test-mem-trace}.

  \begin{cor}[\textbf{Normal congruences as quotients of word graphs}]
    \label{lem-norm-cong-quot}
    There is a one-to-one correspondence between the normal congruences of
    $E(S)$ and the deterministic quotients of $\Gamma_{X}$.
  \end{cor}

  The trace $\Tr(\rho)$ of a congruence $\rho = R ^{\sharp}$ on an inverse
  semigroup $S$ can therefore be computed by:
  \begin{enumerate}[label=(T\arabic*)]
    \item \label{item-T1}
      computing the set $R'$ from \cref{lem-gen-pairs-trace};
    \item\label{item-T2}
      find the greatest quotient of $\Gamma_{X}$ containing $R'$ using the
      variant of the Todd-Coxeter Algorithm described in Section 5
      of~\cite{Coleman2024}.
  \end{enumerate}

  Next we consider an example to illustrate the steps \ref{item-T1} and
  \ref{item-T2}. Each element of a finite symmetric inverse monoid is
  expressible
  as a product of chains and disjoint cycles. So we write \((i_1, \ldots, i_n)\)
  for a cycle and \([i_1, \ldots, i_n]\) for a chain. When points are fixed we
  write \((i)\) to denote that \(i\) is fixed as omitted points are not in the
  domain of the described partial permutation.

  \begin{ex}\label{ex-trace}
    In this example we show how to compute the trace of the least congruence
    $\rho$ on the symmetric inverse monoid $I_4$ (consisting of all the partial
    permutations on the set $\{1, 2, 3, 4\}$) containing the pair:
    \[
      (a, b) := \left(
        (1)(2)(3),
        (1\ 2\ 3)
      \right) \in I_4\times I_4.
    \]
    We use the following generating set for $I_4$:
    \[
      X := \left\{
        x_1 :=  (1\ 2\ 3\ 4),\quad
        x_2 :=  (1\ 2)(3)(4),\quad
        x_3 :=  [4\ 3\ 2\ 1]
      \right\}.
    \]
    If $N$ is the set of generating pairs for $\Tr(\rho)$ from
    \cref{lem-gen-pairs-trace}, then a maximal subset $M$ of $N$ such that
    $M\cap M ^{-1}= \varnothing$ is:
    \[
      M :=
      \left\{
        ((1), (2)), ((1),(3)), ((2), (3)), ((1)(2), (2)(3)), ((1)(2),
      (1)(3)), ((2)(3), (1)(3))\right\}.
    \]
    Obviously, $M$ also generates $\Tr(\rho)$. A diagram of the word graph
    $\Gamma_X$ in this example can be seen in \cref{fig-trace-1}. A diagram of
    the greatest quotient of $\Gamma_{X}$ containing $(a, b)$ is shown in
    \cref{fig-trace-2}.

    \begin{figure}
      \centering
      \includegraphics{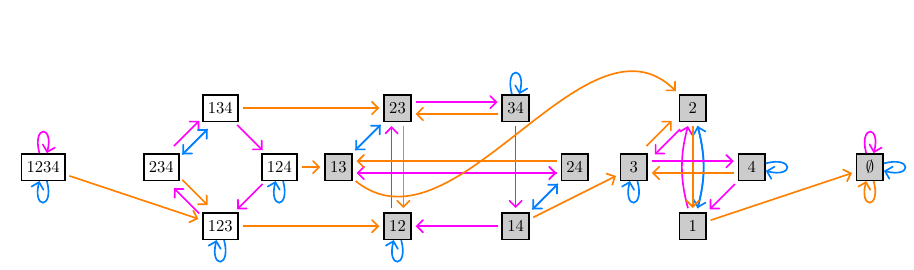}
      \caption{Diagram of the word graph $\Gamma_X$ from \cref{ex-trace}, where
        \(x_1\) is represented in magenta, \(x_2\) in blue and
        \(x_3\) in orange.
        Each node is the idempotent that is the identity on the set
        in its label.
        Shaded nodes correspond to the idempotents belonging to the only
      non-singleton class of $\Tr(\rho)$.}
      \label{fig-trace-1}
    \end{figure}

    \begin{figure}
      \centering
      \includegraphics{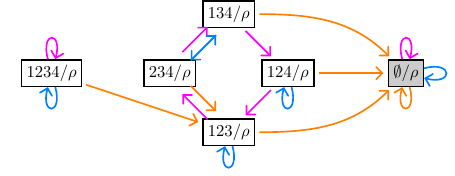}
      \caption{Diagram of the maximum quotient of the word graph $\Gamma_X$ from
        \cref{ex-trace} by the generating pairs of $\Tr(\rho)$ from
        \cref{lem-gen-pairs-trace}, where \(x_1\) is represented in
      magenta, \(x_2\) in blue and \(x_3\) in orange.}
      \label{fig-trace-2}
    \end{figure}
  \end{ex}

  \section{Computing the group $\mathscr{H}$-classes of the quotient}
  \label{section-compute-the-groups}

  In this section we show how to compute the group $\mathscr{H}$-class
  component~\ref{item-Q4} of the quotient data structure.

  We will repeatedly make use of the following simple lemma, which we record for
  the sake of completeness.
  \begin{lem}\label{lem-eliminate-idempotents}
    If $S$ is finite, $e\in E(S)$, $x,y,z\in S$, and $(zxey, z)\in\J$,
    then $zxey=zxy$.
  \end{lem}
  \begin{proof}
    Via the Vagner-Preston Representation Theorem (\cite[Theorem
    5.1.7]{Howie}) we may assume without loss of generality that $S$
    is an inverse
    subsemigroup of the symmetric inverse monoid $I_n$ for some
    non-negative integer $n$.
    If $x\in I_n$, then we denote the number of points in the domain
    (and image) of the function $x$, by $\rank(x)$.
    Since $(z, zxey)\in \J$, it follows that $\rank(z) = \rank(zxey) \leq
    \rank(zxy)\leq \rank(z)$, yielding equality throughout. In particular,
    $\rank(zxey) = \rank(zxy)$, and since $e$ is an idempotent and $S$ is
    finite, it follows that $zxey = zxy$, as required.
  \end{proof}

  If $s/\rho$ is an idempotent in $S/\rho$, then by Lallement's Lemma there
  exists $e\in E(S)$ such that $e/\rho = s/\rho$, and so  $E(S) \cap s/\rho =
  e/\Tr(\rho)$. We define $f\in E(S)$ to be the meet of $e/\Tr(\rho)$, that is,
  \[
    f = \bigwedge e/\Tr(\rho),
  \]
  and we denote the group $\mathscr{H}$-class of $f$ in $S$ by $G$.

  The following lemma describes the group \(\H\)-classes in the
  quotient \(S / \rho\) in terms of the group \(\H\)-classes in \(S\)
  and a normal
  subgroup.
  \begin{lem}\label{lem-trace-facts}
    Suppose that $s\in S$ is such that $s/\rho$ is an idempotent in
    $S/\rho$. If $e\in E(S)$ is such that $e/\rho = s/\rho$,
    $f = \bigwedge e/\Tr(\rho)$, and
    $N = H_f^S\cap \left(f/\rho\right)$, then the following hold:
    \begin{enumerate}[\rm (a)]
      \item\label{lem-trace-facts-a}
        $N$ is a normal subgroup of $G = H_{f}^S$;
      \item \label{lem-trace-facts-b}
        $f/\rho$ is an inverse subsemigroup of $S$ and $N$ is the
        minimum non-empty ideal of $f/\rho$;
      \item \label{lem-trace-facts-c}
        the group $\mathscr{H}$-class $H_{s/\rho}^{S/\rho}$ is isomorphic to
        $G/N$.
    \end{enumerate}
  \end{lem}
  \begin{proof}
    \begin{enumerate}[wide,
        label=(\alph*),leftmargin=*,labelindent=0pt,itemsep=0.4em]
      \item Let \(g, h \in N\). Since \(f\) is an idempotent \(h^{-1} \in f /
        \rho\) and so \(gh^{-1} \in f^2 / \rho = f / \rho\). Thus \(N\) is a
        subgroup of \(G = H_f^S\). If \(n \in N\) and \(g \in G\), then \(g^{-1}
        n g =_\rho g^{-1} f g = f\) and so \(g^{-1} n g \in f / \rho\) and \(N\)
        is normal.

      \item
        We first show that \(f / \rho\) is an inverse subsemigroup of \(S\). Let
        \(a, b \in f / \rho\). Then \(ab =_\rho f^2 = f\) and \(a^{-1} =_\rho
        f^{-1} = f\) and so \(ab, a^{-1} \in f / \rho\). Thus \(f / \rho\) is an
        inverse subsemigroup of \(S\). Next we show that \(N\) is the minimum
        non-empty ideal of \(f / \rho\). Clearly, since \(f \in N\), \(N\) is
        non-empty.

        We begin by showing that $N$ is a left ideal. Suppose that \(n \in N\)
        and \(a \in f / \rho\). Since $a, n\in f/\rho$ and $f/\rho$ is a inverse
        semigroup, \(an, an(an)^{-1}, (an)^{-1}an\in f/\rho\). It
        follows from the
        minimality of $f$ that \(an(an)^{-1}\geq f\) and \((an)^{-1}an\geq f\).
        On the other hand, \(nf=n\) (because $f$ is the identity of
        the group $N$ and $n\in N$) implying that
        \((an)^{-1}an=(an)^{-1}anf\leq f\). Thus
        \((an)^{-1}an=f\). On the other hand, \(an(an)^{-1} \geq f =
        (an)^{-1}an\) and
        \((an(an)^{-1}, (an)^{-1}an) \in \mathscr{D}\), and so
        \cref{lem-D-leq} implies that \((an)^{-1}an=f\) also. Thus
        \((an, f) \in
        \mathscr{H}^S\) and so \(an \in N\), as required.

        We have shown that \(N\) is a left ideal, by symmetry it is a
        right ideal
        also. It remains to show that $N$ is the minimum ideal of $f/\rho$.
        Every non-empty ideal $I$ of \(f/\rho\) contains an element
        of the form \(fc\in N\) for
        some \(c\in f/\rho\). Thus \(f = fc(fc)^{-1} \in I\) since
        $f$ is the unique
        idempotent in \(N\). Therefore the ideal $I$ contains all of
        \(N\) and so \(N\) is the minimum ideal.

      \item We define
        \[\psi \colon G / N \to H^{S / \rho}_{s / \rho}\quad\text{ by
          }\quad Ng \mapsto g /
        \rho.\]
        To show that \(\psi\) is well-defined, we must show that \(\psi\) maps
        into \(H^{S / \rho}_{s / \rho}\) and that \(\psi\) does not
        depend on the
        choice of coset representative. Let \(Ng \in G / N\). Then \((Ng)\psi
          ((Ng)\psi)^{-1} = (g / \rho) \cdot (g^{-1}/  \rho) = f / \rho = s /
        \rho\) and by symmetry \(((Ng)\psi)^{-1} (Ng)\psi = s / \rho\) and so
        \((Ng)\psi \in H_{s / \rho}^{S / \rho}\). If \(h \in Ng\), then \(h =
        ng\) for some \(n \in N\) and so \(h = ng =_\rho fg = g\). So \((Nh)\psi
        = h / \rho = g /\rho = (Ng)\psi\) and \(\psi\) is well-defined. We will
        next show that \(\psi\) is a homomorphism. Let \(Ng, Nh \in G
        / N\). Then
        \[
          (Ng\cdot Nh)\psi = (Ngh)\psi = gh / \rho = (g / \rho) \cdot (h /
          \rho) = (Ng) \psi \cdot (Nh) \psi,
        \]
        and \(\psi\) is a homomorphism. To show \(\psi\) is injective, let \(Ng,
        Nh \in G / N\) be such that \((Ng)\psi = (Nh)\psi\). Then \(g / \rho = h
        / \rho\) and so \(gh^{-1} \in f / \rho\). Thus \(gh^{-1} \in N\), and it
        follows that \(Ng = Nh\), as required. It remains to show that \(\psi\)
        is surjective. Let \(k / \rho \in H_{s / \rho}^{S / \rho}\).
        Then \(k \in G\) and so \(k / \rho = (Nk) \psi\) and \(\psi\)
        is surjective.\qedhere
    \end{enumerate}
  \end{proof}

  The next lemma is the key result in this section, permitting us to
  express \(N\) in terms of \(R\) and the word graph \(\Gamma_X\) of \(S\),
  and allowing for the efficient computation of \(N\).

  \begin{lem}[\textbf{Generating the normal
    subgroups}]\label{lem-normal-subgroup-generators}
    If the strongly connected component of $f$ in $\Gamma_X$ is $\{e_1 = f, e_2,
    \ldots, e_r\}$ for some $r$, and for every $i$, we choose $s_i\in
    S$ to be the
    label of an $(e_1, e_i)$-path in $\Gamma_X$, then $N=H_f^S\cap
    \left(f/\rho\right)$ is the normal closure of
    \[
      \{fs_iab^{-1}s_i^{-1}\mid (a,b)\in R,\ i\in \{1, \ldots, r\}\} \cap H_f ^S
    \]
    in $H_f^S$.
  \end{lem}
  \begin{proof}
    Let $N'$ denote the normal closure of the set in the statement.

    To show that $N'\subseteq N$, suppose that $i\in \{1, \ldots, k\}$ and $(a,
    b)\in R$ are such that $fs_iab^{-1}s_i^{-1}\in H_f^S$. We must show that
    $fs_iab^{-1}s_i^{-1}\in N$; that is, $fs_iab^{-1}s_i^{-1}=_\rho f$ (this is
      sufficient because, by
      \cref{lem-trace-facts}\ref{lem-trace-facts-a}, \(N\)
    is a normal subgroup of \(G\)). We begin by showing that
    $f = fs_iaa^{-1}s_i^{-1}$. Since $s_iaa^{-1}s_i^{-1}\in E(S)$, it
    follows that
    $fs_iaa^{-1}s_i^{-1}\leq f$. On the other hand, since
    $fs_iab^{-1}s_i^{-1}\in
    H_f^S$, which is a group, it follows that
    \begin{align*}
      f &= (fs_iab^{-1}s_i^{-1})(fs_iab^{-1}s_i^{-1})^{-1} && f\text{
        is the identity
      of }H_f^S \\
      &= fs_iab^{-1}s_i^{-1}s_iba^{-1}s_i^{-1}f\\
      &\leq fs_iaa^{-1}s_i^{-1}f && b^{-1}s_i^{-1}s_ib\in E(S)\\
      &= fs_iaa^{-1}s_i^{-1} && \text{idempotents commute in }S.
    \end{align*}
    It follows that $f = fs_iaa^{-1}s_i^{-1}$, and so, in particular, $f=_\rho
    fs_iaa^{-1}s_i^{-1}$. Since $(a, b)\in R$ we have \(a=_\rho b\), it follows
    that $a^{-1}=_\rho b^{-1}$, and so $fs_iab^{-1}s_i^{-1}=_\rho
    fs_iaa^{-1}s_i^{-1}$. Therefore by the transitivity of $\rho$,
    $fs_iab^{-1}s_i^{-1}=_\rho f$, as required.

    For the converse containment ($N\subseteq N'$), suppose that $g\in N =
    H_f^S\cap (f/\rho)$. Since $f=_\rho g$, there exists an elementary sequence
    \[
      a_0 = f,  \ldots, a_i = p_ib_iq_i, a_{i + 1}=p_ic_iq_i, \ldots, a_n = g
    \]
    where $p_i, q_i\in S$ and $(b_i, c_i)$ or $(c_i, b_i)\in R$ for all $i$. By
    assumption, $a_i=_\rho a_0=f$, or equivalently, $a_i\in f/\rho$ for every
    $i$. Thus, since $f/\rho$ is a subsemigroup of $S$ and $N$ is the minimum
    non-empty ideal of $f/\rho$ (\cref{lem-trace-facts}\ref{lem-trace-facts-b}),
    $fa_if\in N$ for all $i$.

    We will show that $fa_kf \in N'$ for every $k$ by induction. Certainly, $a_0
    = f\in N'$ since $f$  is the identity of $G$ and $N'\leq G$. Assume that
    $fa_kf\in N'$ for all $k \leq i$. To prove that $fa_{i + 1}f\in N'$, it
    suffices to show that
    \[
      (fa_if)(fa_{i+1}f)^{-1}\in N'.
    \]
    But \((fa_if)(fa_{i+1}f)^{-1}\in N\), and is thus \(\mathscr{H}\)-related to
    \(f\), hence
    \begin{align*}
      (fa_if)(fa_{i+1}f)^{-1} &= (fp_ib_iq_if)(fq_i^{-1}c_i^{-1}p_i^{-1}f) \\
      & = fp_ib_ic_i^{-1}p_i^{-1}f &&\text{by \cref{lem-eliminate-idempotents}.}
    \end{align*}
    If we set $t = fp_ib_ic_i^{-1}p_i^{-1}f$, then $t =
    (fa_if)(fa_{i+1}f)^{-1} \in
    N\leq Gs$.

    By assumption $fa_if\in N'\leq G$, and $f$ is the identity of the group $G$.
    Hence $(fa_if)^{-1}(fa_if)=f$ and so
    \[
      f = (fa_if)^{-1}(fa_if) = (b_iq_if)^{-1}(p_i^{-1}fp_i)(b_iq_if).
    \]
    This shows that $f$ and $p_i^{-1}fp_i$ belong to the same strongly connected
    component of $\Gamma_X$, and so there exists $j\in\{1, \ldots,
    k\}$ such that
    $s_j^{-1}fs_j = p_i^{-1}fp_i$. Clearly,
    \[
      s_js_j^{-1}f = fs_js_j^{-1}f = fs_js_j^{-1}s_js_j^{-1}f =
      s_js_j^{-1}fs_js_j^{-1}f = f^ 2 = f,
    \]
    and $fp_ip_i^{-1} = f$\footnote{We may assume without loss of
      generality that
      $S$ is an inverse subsemigroup of $I_n$. Since $fp_ip_i^{-1} =
      f|_{\dom(p_i)}$ and $\rank(f) = \rank(s_js_j^{-1}fs_js_j^{-1}) =
      \rank(s_j^{-1}fs_j) = \rank(p_i^{-1}fp_i)$, it follows that
      $\dom(p_i) \cap
      \dom(f)=\im(p_i^{-1})\cap \dom(f) = \dom(f)$ (otherwise
        $\rank(p_i^{-1}fp_i)
      < \rank(f))$. In other words $\dom(f)\subseteq \dom(p_i)$ and so
      $fp_ip_i^{-1}
    = f|_{\dom(p_i)} = f$.}.

    If $u =  fp_is_j^{-1}f$, then
    \[
      s_jp_i^{-1}\cdot fp_is_j^{-1}f = s_js_j^{-1} fs_js_j^{-1}f =
      s_js_j^{-1}f = f \text{ and }
      fp_is_j^{-1}f\cdot s_jp_i^{-1} = fp_ip_i^{-1}f p_i p_i^{-1} =
      fp_ip_i^{-1} = f.
    \]
    In particular, $(u, f)\in \mathscr{H}$, and so $u\in G$. Since
    $t\in G$ also,
    $u^{-1}tu\in N'$ if and only if $t\in N'$ since $N'$ is a normal subgroup of
    $G$. But
    \begin{align*}
      u^{-1}tu & = (fp_is_j^{-1}f)^{-1}\cdot
      (fp_ib_ic_i^{-1}p_i^{-1}f)\cdot (fp_is_j^{-1}f) \\
      & = fs_j\cdot p_i^{-1}fp_i \cdot b_ic_i^{-1}\cdot p_i^{-1}fp_i
      \cdot s_j^{-1}f \\
      & = fs_j\cdot s_j^{-1}fs_j \cdot b_ic_i^{-1}\cdot s_j^{-1}fs_j
      \cdot s_j^{-1}f && s_j^{-1}fs_j = p_i^{-1}fp_i\\
      & = fs_j b_ic_i^{-1}s_j^{-1}f\in N'.
    \end{align*}
    Hence $t\in N'$, and so $(fa_if)(fa_{i+1}f)^{-1}\in N'$, and so $fa_{i +
    1}f\in N'$, as required. We have shown that $fa_if \in N'$ for all $i$, and
    so, in particular, $fa_nf = fgf = g\in N'$.
  \end{proof}

  The algorithm for computing the normal subgroups component
  \ref{item-Q4} of the
  quotient data structure is:

  \begin{enumerate}[label=(N\arabic*)]
    \item \label{item-N1}
      find one $e\in E(S)$ for every strongly connected component of
      $\Gamma_X/\Tr(\rho)$;
    \item \label{item-N2}
      for each representative $e\in E(S)$ from \ref{item-N1}, set $N$ to be the
      trivial group, and iterate through the generating set given in
      \cref{lem-normal-subgroup-generators} for $H_f^S\cap (f/\rho)$
      where $f$ is
      the meet of $e/\Tr(\rho)$, taking the normal closure of $N$ and each
      generator.
  \end{enumerate}

  Next, we continue the example started in~\cref{ex-trace}, and compute the
  generating sets for the normal subgroups in the quotient
  using~\cref{lem-normal-subgroup-generators}.

  \begin{ex}
    \label{ex-norm-subgroups}
    We compute the normal subgroup component \ref{item-Q4} of the quotient data
    structure for the the least congruence  $\rho$ on the symmetric inverse
    monoid $I_4$  containing the pair:
    \[
      (a, b) := \left(
        (1)(2)(3),
        (1\ 2\ 3)
      \right).
    \]
    Firstly the graph $\Gamma_X/\Tr(\rho)$ given in \cref{fig-trace-2} clearly
    has 3 strongly connected components, and so we are attempting to compute 3
    normal subgroups. If the representatives chosen in \ref{item-N1} are
    $1_{I_4} = (1)(2)(3)(4)$, $(1)(2)(3)$, and $\varnothing$, then the normal
    subgroups found in \ref{item-N2} are the trivial group, $\genset{(1\ 3\
    2)}$, and the trivial group, respectively. Thus the quotient groups are (up
    to isomorphism) the symmetric group on $\{1, 2, 3, 4\}$, the cyclic group of
    order $2$, and the trivial group, respectively.
  \end{ex}

  Given the data structure for the quotient $S/\rho$ from
  \ref{item-Q1},~\ref{item-Q2},~\ref{item-Q3}, and~\ref{item-Q4}, the number of
  congruence classes of $\rho$ can be determined as follows. Suppose that
  $f_1/\Tr(\rho), \ldots, f_r/\Tr(\rho)$ are representatives of the strongly
  connected components $C_1, \ldots, C_r$ of $\Gamma_X/\Tr(\rho)$
  (from~\ref{item-Q3}) for some $r\geq 1$ and for some $f_1, \ldots, f_r\in
  E(S)$. We may assume without loss of generality that each $f_i$ is the least
  (with respect to the natural partial order on $S$) idempotent in its trace
  class $f_i/\Tr(\rho)$. That is, $f_i = \bigwedge f_i/\Tr(\rho)$ for every $i$.
  Then the number of congruence classes of $S$ is:
  \begin{equation}\label{equation-nr-classes}
    \sum_{i = 1}^r |G_i/N_i| |C_i|^2
  \end{equation}
  where $G_i = H_{f_i}^S$ is the group $\H^S$-class of $f_i$ and $N_i =H_{f_i}^S
  \cap \left(f_i/\rho\right)$ is the normal subgroup of $G_i$ from
  \cref{lem-trace-facts}; see~\cite[Section 5.6]{East2019aa} for
  further details.

  We can also determine a set of representatives of the congruence classes of
  $\rho$ from the quotient data structure. Clearly from
  \eqref{equation-nr-classes} there is a one-to-one correspondence between
  congruence classes of $\rho$ and elements of $C_i\times \left(G_i/N_i\right)
  \times C_i$ for $i\in\{1, \ldots, r\}$. Suppose that $C_i = \{e_1 := f_i,
  \ldots, e_{|C_i|}\}$, $s_j\in S$ is any element such that $s_j^{-1}e_1s_j =
  e_j$ for every $j\in \{1, \ldots, |C_i|\}$, and $\{n_1, \ldots,
  n_{|G_i/N_i|}\}$ is a transversal of the cosets of $N_i$ in $G_i$.
  Then, by Green's Lemma, the
  representatives of $\rho$-classes corresponding to $C_i$ are given by
  \[
    \{ s_j^{-1} f_i n_k s_l : 1\leq j, l\leq |C_i|,\ 1\leq k\leq |G_i/N_i|\}.
  \]
  The elements $s_j\in S$ correspond to the products of the labels
  of the edges
  on any path from $e_1$ to $e_j$ in $\Gamma_X/\Tr(\rho)$.

  \begin{ex}
    Continuing \cref{ex-norm-subgroups}, the number of classes of the
    congruence $\rho$ is $4!\cdot1^2 + 2 \cdot 4 ^ 2 + 1 \cdot 1 ^ 2 = 57$.

    We consider the $\rho$-class representatives corresponding to the only
    non-trivial strongly connected component of $\Gamma_X/\Tr(\rho)$
    (see~\cref{fig-trace-2}) where $f = (1)(2)(3)$, $G$ is the symmetric group
    on the set $\{1, 2, 3\}$, and $N = \genset{(1\ 2\ 3)} \trianglelefteq G$.
    We choose the transversal of cosets of $N$ in $G$ to be $\{f, (1)(2\
    3)\}$. The elements $s_j\in S$ are:
    \begin{align*}
      s_1 = (1)(2)(3)(4) && s_2 = (1\ 2\ 3\ 4) \\
      s_3 = (1\ 3)(2\ 4) && s_4 = (1\ 4\ 3\ 2).
    \end{align*}
    The representatives corresponding to the coset representative $f$
    are:
    \[
      \begin{array}{c|c|c|c|c}
        & s_1 & s_2 & s_3 & s_4 \\ \hline
        s_1
        & s_1 ^{-1} f^2 s_1 = f
        & s_1 ^{-1} f^2 s_2 = [1\ 2\ 3\ 4]
        & s_1^{-1} f^2 s_3 = [2\ 4](1\ 3)
        & s_1^{-1}f^2s_4 = [3\ 2\ 1\ 4]\\
        s_2
        & s_2 ^{-1}f^2 s_1 = [4\ 3\ 2\ 1]
        & s_2 ^{-1}f^2 s_2 = (2)(3)(4)
        & s_2 ^{-1}f^2 s_3 = [2\ 3\ 4\ 1]
        & s_2 ^{-1}f^2 s_4 = [3\ 1](2\ 4) \\
        s_3
        & s_3 ^{-1}f^2 s_1 = [4\ 2](1\ 3)
        & s_3 ^{-1}f^2 s_2 = [1\ 4\ 3\ 2]
        & s_3 ^{-1}f^2 s_3 = (1)(3)(4)
        & s_3 ^{-1}f^2 s_4 =  [3\ 4\ 1\ 2] \\
        s_4
        & s_4 ^{-1}f^2 s_1 = [4\ 1\ 2\ 3]
        & s_4 ^{-1}f^2 s_2 = [1\ 3](2\ 4)
        & s_4 ^{-1}f^2 s_3 = [2\ 1\ 4\ 3]
        & s_4 ^{-1}f^2 s_4 =  (1)(2)(4)
      \end{array}
    \]
    and for the coset representative $n := (1)(2\ 3)$:
    \[
      \begin{array}{c|c|c|c|c}
        & s_1 & s_2 & s_3 & s_4 \\ \hline
        s_1
        & s_1 ^{-1} fn s_1 = (1)(2\ 3)
        & s_1 ^{-1} fn s_2 =[1\ 2\ 4](3)
        & s_1^{-1} fn s_3 =[2\ 1\ 3\ 4]
        & s_1^{-1}fns_4 = [3\ 1\ 4](2)\\
        s_2
        & s_2 ^{-1}fn s_1 =[4\ 2\ 1](3)
        & s_2 ^{-1}fn s_2 =(2)(3\ 4)
        & s_2 ^{-1}fn s_3 = [2\ 3\ 1](4)
        & s_2 ^{-1}fn s_4 = [3\ 2\ 4\ 1]\\
        s_3
        & s_3 ^{-1}fn s_1 =[4\ 3\ 1\ 2]
        & s_3 ^{-1}fn s_2 =[1\ 3\ 2](4)
        & s_3 ^{-1}fn s_3 =(1\ 4)(3)
        & s_3 ^{-1}fn s_4 = [3\ 4\ 2](1)\\
        s_4
        & s_4 ^{-1}fn s_1 =[4\ 1\ 3](2)
        & s_4 ^{-1}fn s_2 =[1\ 4\ 2\ 3]
        & s_4 ^{-1}fn s_3 =[2\ 4\ 3](1)
        & s_4 ^{-1}fn s_4 = (1\ 2)(4).
      \end{array}
    \]
  \end{ex}

  \section{Computing a class of a congruence}\label{section-compute-the-kernel}

  In this section we consider the question of how to enumerate the elements of
  $x/\rho$ for an arbitrary $x\in S$. This provides a means of enumerating the
  elements of the kernel $\Ker(\rho)$ of the congruence $\rho$: find
  representatives $e_1, \ldots, e_{k}\in E(S)$ of the congruences classes of
  $\Tr(\rho)$, and then apply the algorithm in this section to determine the
  elements of $e_i/\rho$ for every $i$.

  Throughout this section we fix $x\in S$ with the aim of enumerating $x/\rho$.
  We require the following definitions, which are the key ingredients in this
  section:
  \[
    U_x = \bigcup D_{x / \rho}^{S / \rho} = \set{y\in S}{(y/\rho,
    x/\rho)\in \D^{S/\rho}},
  \]
  \(\mu, \nu \colon S \to E(S)\) defined by
  \[
    (y) \mu = \min((yy^{-1})/\Tr(\rho)) \quad\text{ and }\quad (y)\nu = \min
    ((y^{-1}y)/ \Tr(\rho)),
  \]
  and, finally, $\phi_x \colon U_x \to S$ defined by
  \[
    (y)\phi_x = (y) \mu \cdot y \cdot (y) \nu.
  \]

  The following lemma collects various properties of \(U_x\) and \(\phi_x\) that
  are used repeatedly throughout this section.

  \begin{lem}\label{lem-phi-1}
    \begin{enumerate}[\rm(a)]
      \item \label{item-94}
        If $y\in S$, then $(y)\mu = (yy^{-1})\mu$;
      \item \label{item-96}
        If $y, z\in S$ and $y=_\rho z$, then $(y)\mu = (z)\mu$;
      \item \label{part-a}
        If $y, z\in U_x$ and \(y=_\rho z\), then \(y=_\rho
        (y)\phi_x=_\rho (z)\phi_x\);
      \item \label{part-h}
        If $y, z\in U_x$ are such that $(y)\mu =_{\rho}(z)\mu$,
        then $(y)\mu = (z)\mu$. Similarly,
        if $(y)\nu =_{\rho}(z)\nu$, then $(y)\nu = (z)\nu$;
      \item \label{part-b}
        $\phi_x \circ \phi_x = \phi_x$ and, in particular,
        $\im(\phi_x)\subseteq U_x$;
      \item \label{part-c}
        If $y, z\in U_x$, then $(y/\rho, z/\rho)\in \D ^ {S/\rho}$;
      \item \label{part-d}
        If $y\in U_x$ and $(y, z)\in \D^S$, then $z \in U_x$;
      \item \label{part-e}
        If $y\in U_x$, then $(y)\phi_x \leq y$;
      \item \label{part-f}
        If $y, z\in U_x$ and $y\leq z$, then $y=_\rho z$;
      \item \label{item-95}
        If $y, z\in U_x$ and $y\leq z$, then $(y)\mu = (z)\mu$;
      \item \label{part-g}
        If $y\in U_x$, then $(y)\phi_x\cdot ((y)\phi_x)^{-1} = (y)\mu$ and
        $((y)\phi_x)^{-1}\cdot (y)\phi_x = (y)\nu$;
      \item \label{part-i}
        If \(y, z, yz \in U_x\) then \((y)\phi_x \cdot (z)\phi_x =
        (y)\phi_x z\).
    \end{enumerate}
  \end{lem}
  \begin{proof}
    \begin{enumerate}[wide,
        label=(\alph*),leftmargin=*,labelindent=0pt,itemsep=0.4em]
      \item We have
        \(\min((yy^{-1})/\Tr(\rho))=\min((yy^{-1})(yy^{-1})^{-1}/\Tr(\rho))\).
      \item We
        have\((y)\mu=\min((yy^{-1})/\Tr(\rho))=\min((yy^{-1})/\rho)=\min((y/\rho)(y/\rho)^{-1})\)
        so depends only on the \(\rho\) class of \(y\).
      \item

        Suppose that $y=_\rho z$. Then by definition $(y)\phi_x = (y)\mu \cdot y
        \cdot (y)\nu$. Since $(y)\mu=_\rho yy^{-1}$ and $ (y)\nu=_\rho y^{-1}y$,
        it follows that $(y)\phi_x=(y)\mu \cdot y \cdot (y)\nu=_\rho
        yy^{-1}\cdot
        y \cdot y^{-1}y = y$

        Similarly \((z)\phi_x=_\rho z\), hence \((z)\phi_x=_\rho z =_\rho
        y =_\rho (y)\phi_x\) as required.

      \item
        By definition both $(y)\mu$ and $(z)\mu$ are the minimum elements in
        their trace classes. Since $(y)\mu =_{\rho}(z)\mu$, these trace classes
        coincide, and so $(y)\mu =  (z)\mu$. The proof for $\nu$ is the
        same.

      \item  Let $y\in \im(\phi_x)$. Then there exists $z\in U_x$ such that
        $(z)\phi_x = y$. By part \ref{part-a}, $(z)\phi_x=_\rho z$
        and so $y=(z)\phi_x
        =_\rho z$. In particular, $y/\rho = z/\rho$ and so $(y/\rho, x/\rho) =
        (z/\rho, x/\rho) \in \mathscr{D}^{S/\rho}$, since $z\in U_x$. This shows
        that $y\in U_x$.

        By part \ref{part-h}, $y=_\rho z$ implies that $(y)\mu = (z)\mu$ and
        $(y)\nu=(z)\nu$. In particular, since $(y)\mu=(z)\mu$ and $(y)\nu=
        (z)\nu$ are idempotents, $(z)\mu = (y)\mu\cdot (z)\mu$ and $(z)\nu =
        (z)\nu\cdot (y)\nu$. Hence
        \begin{align*}
          y=(z)\phi_x& =(z)\mu\cdot z \cdot (z)\nu \\
          & = (y)\mu\cdot (z)\mu \cdot z \cdot (z)\nu \cdot (y)\nu \\
          & = (y)\mu \cdot (z)\phi_x \cdot (y)\nu \\
          & = (y)\mu \cdot y \cdot (y)\nu \\
          & = (y)\phi_x.
        \end{align*}

      \item
        If $y, z\in U_x$, then $(y/\rho, x/\rho), (z/\rho, x/\rho)\in \D
        ^{S/\rho}$ by definition. Thus, since $\rho$ is transitive, $(y/\rho,
        z/\rho)\in \D ^{S/\rho}$.

      \item
        If $(y, z)\in \D ^{S}$, then $(y/\rho, z/\rho)\in \D ^ {S/\rho}$. Since
        $y\in U_x$, $(y/\rho, x/\rho)$, and so, again by the transitivity of
        $\rho$, $(z/\rho, x/\rho)\in \D ^ {S/\rho}$. Therefore $z\in U_x$.

      \item
        This follows immediately from the definition of $\phi_x$, since
        $(z)\phi_x = (z)\mu\cdot z\cdot (z)\nu \leq z$.

      \item
        Since $y\leq z$, it follows that $x/\rho \leq y/\rho$. But by assumption
        $y, z\in U_x$ and so $(y/\rho, z/\rho)\in \D ^ {S/\rho}$, by
        part \ref{part-c}.
        Therefore $y/\rho =z/\rho$, as required.
      \item This follows from part \ref{item-96} and part \ref{part-f}.
      \item
        By part \ref{part-a}, $(y)\phi_x((y)\phi_x)^{-1}=_{\rho} yy^{-1}$ and so
        $(y)\phi_x((y)\phi_x)^{-1} \geq (y)\mu$, since $(y)\mu$ is
        the minimum of
        its trace class. On the other hand, $(y)\mu\cdot (y)\phi_x = (y)\phi_x$
        and so $(y)\mu (y)\phi_x((y)\phi_x)^{-1} = (y)\phi_x((y)\phi_x)^{-1}$.
        Therefore $(y)\phi_x((y)\phi_x)^{-1} \leq (y)\mu$.
        The proof for $\nu$ follows by symmetry.

      \item
        By part \ref{part-e}, \((y)\phi_x \cdot (z)\phi_x \leq
        (y)\phi_x \cdot z\).
        Part \ref{item-95}  then implies that
        $((y)\phi_xz)\mu = ((y)\phi_x(z)\phi_x)\mu$.
        From the definition of \(\phi_x\), we have \((y)\phi_x=(y)\mu
        (y)\phi_x\).
        By part \ref{item-94}
        $((y)\mu(y)\phi_xz)\mu=((y)\mu(y)\phi_xzz^{-1}((y)\phi_x)^{-1}(y)\mu)\mu$.
        By part \ref{item-95} this is in turn equal to
        \(((y)\mu)\mu\) which from the definition of \(\mu\) is equal
        to \((y)\mu\). By the same argument, \((y)\mu=((y)\phi_x(z)\phi_x)\mu\).
        So $(y)\mu=((y)\phi_xz)\mu = ((y)\phi_x(z)\phi_x)\mu$. Thus
        from the definition of \(\mu\), \((y)\mu\) is below each of
        \((y)\phi_xz((y)\phi_xz)^{-1}\) and
        \(((y)\phi_x(z)\phi_x)((y)\phi_x(z)\phi_x)^{-1}\). From the
        definition of \(\phi_x\), it thus follows that
        \[(y)\phi_xz((y)\phi_xz)^{-1}=((y)\phi_x(z)\phi_x)((y)\phi_x(z)\phi_x)^{-1}=(y)\mu.\]
        As \((y)\phi_x \cdot (z)\phi_x \leq (y)\phi_x \cdot z\),
        \[(y)\phi_x \cdot
          (z)\phi_x=((y)\phi_x(z)\phi_x)((y)\phi_x(z)\phi_x)^{-1}
          (y)\phi_x \cdot z=(y)\phi_xz((y)\phi_xz)^{-1}(y)\phi_x
          \cdot z=(y)\phi_x \cdot z.
        \qedhere\]
    \end{enumerate}
  \end{proof}

  For the next lemma it will be convenient to use the languages of
  groupoids. The
  set $U_x$ naturally forms a groupoid with $\ast\colon U_x\times U_x \to U_x$
  defined by
  \[
    y \ast z = yz \quad \text{whenever } y, z, yz\in U_x \text{ and }
    yz\D ^S y \D ^ S z
  \]
  and where the inverse operation coincides with that on $S$. The connected
  components of $U_x$ are just the $\D$-classes of $S$; for further details see
  \cite[Section 3.1]{Lawson1998}.

  \begin{lem}
    \label{lem-functor}
    If $u\in U_x$, then $\phi_x|_{D_u^S}$ is a functor (or
    equivalently a groupoid morphism).
  \end{lem}
  \begin{proof}
    Suppose that $y, z\in D_u^S$ are such that $yz\in D_u^S$. It
    suffices to show
    that $(y)\phi_x\cdot (z)\phi_x = (yz)\phi_x$. Since $yz\D ^S u\D
    ^S y$, $y^{-1}y = z
    z ^ {-1}$ (by the Location \cref{thm-product-location}), and so $(z)\mu=
    (y)\nu$. It follows that
    \begin{align*}
      (y)\phi_x\cdot (z)\phi_x & = ((y)\mu \cdot y\cdot
      (y)\nu)((z)\mu \cdot z\cdot (z)\nu) \\
      & = (y)\mu \cdot y\cdot (y)\nu\cdot z\cdot (z)\nu &&
      \text{since }(y)\nu = (z)\mu\in E(S) \\
      & = (y)\mu \cdot yz\cdot (z)\nu && \text{by
      \cref{lem-eliminate-idempotents}}.
    \end{align*}
    Since $(yz)(yz) ^{-1} = yzz^{-1}y^{-1} = yy^{-1}yy^{-1} = yy
    ^{-1}$, it follows that
    $(y)\mu = (yz)\mu$ and similarly, $(z)\nu = (yz)\nu$. Therefore
    \[
      (y)\phi_x\cdot (z)\phi_x = (y)\mu \cdot yz\cdot (z)\nu  =
      (yz)\mu\cdot yz \cdot
      (yz)\nu = (yz)\phi_x.\qedhere
    \]
  \end{proof}

  \begin{lem}
    \label{lem-im-phi-minimal}
    \begin{enumerate}[\rm(a)]
      \item If $y\in \im(\phi_x)$ and $z\in U_x$ is such that $z\leq
        y$, then $y = z$.

      \item If $e=_{\Tr(\rho)} f$, $e\in \im(\phi_x)$ is an
        idempotent, and $f\in
        U_x$ is also an idempotent, then $e\leq f$.
    \end{enumerate}
  \end{lem}

  \begin{proof}
    Suppose that $y\in \im(\phi_x)$ and $z\in U_x$ are such that $z\leq y$.
    If $z < y$, then $zz^{-1} < yy ^{-1}$. Hence it suffices to show
    that $yy^{-1}$ is
    minimal in $U_x$.

    Since $y\in \im(\phi_x)$, \cref{lem-phi-1}(a) shows that
    $(y)\phi_x = y\in U_x$. We
    begin by showing that $yy^{-1}$ is the minimum in its trace
    class; this will establish
    part (b). By the definition of $\mu$, it suffices to show that $yy^{-1} =
    (u)\mu$. This follows from \cref{lem-phi-1}(g).

    By the definition of $\phi_x$, $y = (u)\mu\cdot u\cdot (u)\nu$ and so
    \[
      y y ^ {-1} = ((u)\mu\cdot u\cdot (u)\nu)((u)\mu\cdot u\cdot (u)\nu)^{-1}
      = (u)\mu\cdot u\cdot (u)\nu \cdot ((u)\nu)^{-1} \cdot u^{-1}
      \cdot ((u)\mu)^{-1}
      \leq (u)\mu \cdot (u)\mu ^ {-1} = (u)\mu,
    \]
    the last equality holds because $(u)\mu$ is an idempotent.

    For the converse inequality, by \cref{lem-phi-1}(a), $y=
    (u)\phi_x=_\rho u$, and so $yy^{-1}=_\rho uu^{-1}=_\rho(u)\mu$.
    So, $yy^{-1}=_\rho
    (u)\mu$, and since $(u)\mu$ is the minimum in its trace class, $yy^{-1}
    \geq (u)\mu$. We have shown that $yy^{-1} = (u)\mu$, meaning that
    $yy^{-1}$ is
    the minimum in its trace class.

    If $z\in U_x$ and $z \leq yy^{-1}$, then $z$ is an idempotent,
    and $z/\rho \leq
    yy^{-1}/\rho$ (homomorphisms preserve the natural partial order).
    Since $y, z\in
    U_x$, it follows that $(y/\rho, z/\rho)\in \D^{S/\rho}$, by
    \cref{lem-phi-1}(c).
    Since $(y, yy^{-1})\in \D^{S}$, and homomorphisms preserve Green's
    $\D$-relation, $(y/\rho, yy^{-1}/\rho)\in \D ^{S/\rho}$. Thus
    $(yy^{-1}/\rho,
    z/\rho)\in \D^{S/\rho}$ and $z/\rho \leq yy^{-1}/\rho$, which implies that
    $z/\rho = yy^{-1}/\rho$.  Hence $z = yy^{-1}$ by
    \cref{lem-phi-1}(h) and so
    $yy^{-1}$ is minimal in $U_x$, as required.
  \end{proof}

  \begin{lem}
    The set $\im(\phi_x)$ is a $\D^{S}$-class.
  \end{lem}
  \begin{proof}
    Suppose that $y\in \im(\phi_x)$. We will show that $D_y^S = \im(\phi_x)$.

    ($\supseteq$) Suppose that $z\in \im(\phi_x)$. Then $y, z\in U_x$
    (\cref{lem-phi-1}(b)) and
    so $(y/\rho, z/\rho)\in \D^{S/\rho}$ (\cref{lem-phi-1}(c)). Hence
    there exists
    $s\in S$ such that $(s/\rho, y/\rho)\in \D^{S/\rho}$, $zz^{-1}
    =_{\rho} ss^{-1}$,
    and $yy^{-1}=_{\rho} s^{-1}s$. This implies that $s\in U_x$ and so
    $(s)\phi_x \in \im(\phi_x)$. We will show that $(s)\phi_x
    ((s)\phi_x)^{-1} =
    zz^{-1}$ and $((s)\phi_x)^{-1}(s)\phi_x = yy^{-1}$. It will
    follow from this
    that $(z, (s)\phi_x), ((s)\phi_x, y)\in \D ^ S$ implying that
    $(z, y)\in \D ^S$
    which will conclude the proof.

    By \cref{lem-phi-1}(a), $s=_\rho (s)\phi_x$ and so
    $(s)\phi_x((s)\phi_x)^{-1} =_{\rho} ss ^{-1} =_{\rho} zz ^{-1}$.
    Since $z\in \im(\phi_x) \subseteq U_x$, it follows from the
    definition of $U_x$
    that $zz ^ {-1}\in U_x$ also.
    \cref{lem-phi-1}(b) implies that $z = (z)\phi_x$ and since $\phi_x$
    is a functor (\cref{lem-functor}), $(zz^{-1})\phi_x= (z)\phi_x\cdot
    (z^{-1})\phi_x = (z)\phi_x\cdot ((z)\phi_x) ^{-1} = zz^{-1}$ (the
      second to last equality holds
    because functors preserve inverses). Hence $zz^{-1}\in \im(\phi_x)$, and
    similarly, $(s)\phi_x ((s)\phi_x)^{-1}\in \im(\phi_x)$. But
    $(s)\phi_x((s)\phi_x)^{-1} =_{\rho} zz ^{-1}$, and so
    \cref{lem-phi-1}(g,h) implies that $zz^{-1}=(s)\phi_x
    ((s)\phi_x)^{-1}$. By symmetry $yy ^{-1} =
    ((s)\phi_x)^{-1}(s)\phi_x$, as required.

    ($\subseteq$) If $z\in D_y^S$, then $z\in U_x = \dom(\phi_x)$, by
    \cref{lem-phi-1}(d). It
    follows that $(z)\phi_x\leq z$ (by \cref{lem-phi-1}(e)) and so
    (assuming without loss of
    generality that $S$ is an inverse semigroup of partial permutations)
    \[
      \rank((z)\phi_x) \leq \rank(z) = \rank(y) = \rank((y)\phi_x) =
      \rank((z)\phi_x),
    \]
    (the last equality holds since $\im(\phi_x)\subseteq D_y^S$).
    Hence $(z)\phi_x = z$ and so $z\in \im(\phi_x)$.
  \end{proof}
  If $y, z\in S$ and $(y, z)\in \D^S$, then in the following results we
  will denote the
  intersection of the $\R$-class $R_y^S$ of $y$ and the $\L$-class
  $L_z^S$ of $z$
  by $H_{y, z}$.

  We can finally state and prove the main result in this section
  which will allow
  us to compute the elements in the congruence class $x/\rho\cap H_{e, f}$ where
  $e, f\in E(S)$, as a translate of the preimage of a coset of a normal subgroup
  under the functor $\phi_x$.

  \begin{theorem}\label{thm-preimages}
    If $x\in S$ is arbitrary and $e, f\in E(S)$ are such that $(e,
    f)\in \D^S$ and $H_{e, f} \cap x/\rho\neq\varnothing$,
    then
    \[
      H_{e, f} \cap x/\rho = \left(\left(H_{(e)\phi_x, (e)\phi_x}\cap
      e/\rho\right)xs^{-1}\right)
      \phi_x|_{H_{e, e}}^{-1}\cdot s,
    \]
    for every $s\in S$ such that $s^{-1}es = f$.
  \end{theorem}

  Before giving the proof of \cref{thm-preimages} note that $H_{(e)\phi_x,
  (e)\phi_x}\cap e/\rho$ is a normal subgroup of $H_{(e)\phi_x, (e)\phi_x}$ by
  \cref{lem-trace-facts}(a). Hence $\left(H_{(e)\phi_x, (e)\phi_x}\cap
  e/\rho\right)xs^{-1}$ is a coset of a normal subgroup. (Although the
    representative of this coset is $(e)\phi_x xs^{-1}$ not necessarily
  $xs^{-1}$.)
  Since $H_{e, e}$ is a group, and $\phi_x$ is a functor
  (\cref{lem-functor}), it
  follows that $\phi_x|_{H_{e, e}}\colon H_{e, e} \to H_{(e)\phi_x, (e)\phi_x}$
  is a group homomorphism. 

  \begin{proof}[Proof of \cref{thm-preimages}]
    Suppose that $s\in S$ is any element such that $s^{-1}es = f$
    (such an element
    exists because $(e, f)\in \D^S$). We start by noting that:
    \[
      fs^{-1}s = s^{-1}ess^{-1}s = s^{-1}es= f,
    \]
    which will be useful in both parts of the proof below.

    $(\subseteq)$
    Let $t\in H_{e, f}\cap x/\rho$. Since $t\in H_{e, f}$, $ts^{-1}s
    = t$ (Green's
    Lemma~\cite[Lemma 2.2.1]{Howie}) and so $ts^{-1}\in H_{e, e}$. In
    particular,
    $(ts^{-1}, e)\in \D^S$ and $e =_{\rho}xx^{-1} \in U_x$ and so $e
    \in U_x$. Thus
    $ts^{-1}\in U_{x}$ (\cref{lem-phi-1}(d)) and so it suffices to show that
    $(ts^{-1})\phi_x \in \left(H_{(e)\phi_x, (e)\phi_x}\cap
    e/\rho\right)xs^{-1}$.

    We start by showing that $(ts^{-1})\phi_x\cdot sx^{-1} =_{\rho} e$:
    \begin{align*}
      (ts^{-1})\phi_x \cdot sx^{-1}
      & =_{\rho} (ts^{-1})sx^{-1} && ts^{-1} =_{\rho} (ts^{-1})\phi_x
      \text{ by \cref{lem-phi-1}(a)}\\
      & =_{\rho} (xs^{-1})sx^{-1} && t=_{\rho} x \text{ by assumption}\\
      & =\phantom{_{\rho}} (xx^{-1}xs^{-1})sx^{-1}  \\
      & =_{\rho} (xfs^{-1}ss^{-1})sx^{-1} && fs^{-1}s =_{\rho} x^{-1}x\\
      & =\phantom{_{\rho}} xfs^{-1}sx^{-1}  \\
      & =\phantom{_{\rho}} xfx^{-1} && fs^{-1}s = f\\
      & =_{\rho} xx^{-1}xx^{-1} && f=_{\rho} x^{-1}x \\
      & =\phantom{_{\rho}} xx^{-1}\\
      & =_{\rho} e.
    \end{align*}
    It remains to show that $(ts^{-1})\phi_x \in H_{(e)\phi_x,
    (e)\phi_x}xs^{-1}$,
    or, by Green's Lemma, equivalently that $(ts^{-1})\phi_x sx^{-1}\in
    H_{(e)\phi_x, (e)\phi_x}$.

    In order to do this, we start by showing that $sx^ {-1}\in U_x$. Since
    $(sx^{-1}, sx^{-1}xs^{-1})\in \D^S$, it follows that $(sx^{-1}/\rho,
    sx^{-1}xs^{-1}/\rho)\in \D ^{S/\rho}$. But $sx^{-1}xs^{-1}=_{\rho}sfs^{-1}=
    ss^{-1}ess^{-1}=ess^{-1}=e=_{\rho}xx^{-1}$ and so $(sx^{-1}xs^{-1}/\rho,
    x/\rho)\in \D ^{S/\rho}$. By transitivity, $(sx^{-1}/\rho, x/\rho)\in \D
    ^{S/\rho}$ and so $sx^{-1}\in U_x$ (\cref{lem-phi-1}(d)). By the
    definitions of
    $\mu$  and $\nu$:
    \begin{equation}\label{eq-ts=sx}
      (ts^{-1})\nu =_{\rho} (ts^{-1})^{-1}ts^{-1} = st^{-1}ts
      =_{\rho} sx^{-1}xs^{-1}
      = sx^{-1}(sx^{-1})^{-1}=_{\rho} (sx^{-1})\mu
    \end{equation}
    so $(ts^{-1})\nu=(sx^{-1})\mu$ (\cref{lem-phi-1}\ref{part-h}). On
    the other hand,
    $ts^{-1}st^{-1}\leq tt^{-1}$ and $tt^{-1}\in U_x$ since $t=_{\rho}x$.
    If $ts^{-1}st^{-1}\in U_x$, then $(ts^{-1})\mu =
    ts^{-1}st^{-1}=_{\Tr(\rho)} tt^{-1}=(t)\mu$
    (\cref{lem-phi-1}(f)). To show that $ts^{-1}st^{-1}\in U_x$ it
    suffices to show that $ts^{-1}st^{-1}$ is
    $\rho$-related to an element of $U_x$: 
    \begin{align*}
      ts^{-1}st^{-1} & = ts^{-1}st^{-1}tt^{-1} \\
      & =_{\rho} ts^{-1}sx^{-1}xt^{-1} && t^{-1}t =_{\rho}x^{-1}x\\
      & =_{\rho} ts^{-1}sft^{-1} && x^{-1}x=_{\rho} f \\
      & = ts^{-1}ss^{-1}est^{-1} && s^{-1}es=f \\
      & = ts^{-1}est^{-1} \\
      & =_{\rho} tft^{-1} &&  s^{-1}es=f \\
      & =_{\rho} tx^{-1}xt^{-1} &&   x^{-1}x=_{\rho} f \\
      & =_{\rho} xx^{-1}xx^{-1} && t=_{\rho}x \\
      & =_{\rho} xx^{-1}\in U_x.
    \end{align*}
    Hence $(ts^{-1})\mu=(t)\mu$ (\cref{lem-phi-1}(h)). By the assumption at the
    start of the proof, $t=_{\rho} x$ and so $tt^{-1}=_{\rho}xx^{-1}$, and so
    $(t)\mu =(xx^{-1})\mu$. Since $e\in E(S)$, $(e)\phi_x = (e)\mu
    =_{\rho}(xx^{-1})\mu$, and again by \cref{lem-phi-1}(h), $(e)\mu =
    (xx^{-1})\mu$. We have shown that
    \begin{equation}\label{eq-ts}
      (ts^{-1})\mu = (e)\phi_x.
    \end{equation}
    By a similar argument,  $(sx^{-1})\nu = (x^{-1})\nu = (x)\mu = (e)\phi_x$.

    It follows that
    \begin{align*}
      (ts^{-1})\phi_x \cdot sx^{-1} & = (ts^{-1})\phi_x \cdot
      (ts^{-1})\nu\cdot sx^{-1} &&
      \text{by the definition of }\phi_x\\
      & = (ts^{-1})\phi_x \cdot (sx^{-1})\mu\cdot sx^{-1} &&\text{by
      }\eqref{eq-ts=sx}\\
      & = (ts^{-1})\phi_x \cdot (sx^{-1})\mu\cdot
      sx^{-1}\cdot(sx^{-1})\nu && \text{by \cref{lem-eliminate-idempotents}}\\
      & = (ts^{-1})\phi_x \cdot (sx^{-1})\phi_x&& sx^{-1}\in U_x.
    \end{align*}
    We set $a = (ts^{-1})\phi_x$ and $b =  (sx^{-1})\phi_x$. By
    \eqref{eq-ts=sx},
    $(ts^{-1})\nu = (sx^{-1})\mu$ and so by \cref{lem-phi-1}(g),
    $a^{-1}a = bb^{-1}$
    and so $(ab, a) \in \D^S$. The Location \cref{thm-product-location}
    then implies that $ab\in
    H_{aa^{-1}, b^{-1}b}$. But
    \begin{align*}
      aa^{-1} & = (ts^{-1})\phi_x ((ts^{-1})\phi_x) ^{-1} \\
      & = (ts^{-1})\mu && \text{by \cref{lem-phi-1}(g)}\\
      & = (e)\phi_x && \text{by } \eqref{eq-ts}.
    \end{align*}
    Similarly, $b^{-1}b = (sx^{-1})\nu = (e)\phi_x$. Whence
    $(ts^{-1})\phi_x \cdot
    sx^{-1} = (ts^{-1})\phi_x \cdot (sx^{-1})\phi_x = ab\in
    H_{aa^{-1}, b^{-1}b} =
    H_{(e)\phi_x, (e)\phi_x}$, as required.
    \medskip


    $(\supseteq)$ Let \(t \in \left(\left(H_{(e)\phi_x, (e)\phi_x}\cap
    e/\rho\right)xs^{-1}\right) \phi_x|_{H_{e, e}}^{-1}\cdot s\) be
    arbitrary. We
    must show that $t\in H_{e, f}$ and $t\in x/\rho$.

    Since $t\in \dom(\phi_x|_{H_{e, e}})s$, there exists $h\in H_{e,
    e}$ such that
    $t = hs\in H_{e, e}s = H_{e, f}$.

    It remains to prove that $t=_{\rho} x$:
    \begin{align*}
      t & = hs && \text{where }h\in \left(\left(H_{(e)\phi_x,
      (e)\phi_x}\cap e/\rho\right)xs^{-1}\right)
      \phi_x|_{H_{e, e}}^{-1}\\
      & =_{\rho} (h)\phi_xs && \text{by \cref{lem-phi-1}(a)}\\
      & =_{\rho} exs^{-1}s && \text{by the choice of }h\\
      & = exx^{-1}xs^{-1}s\\
      & =_{\rho} exfs^{-1}s && f=_{\rho} x^{-1}x\\
      & =exf && fs^{-1}s = f\\
      & =_{\rho} xx^{-1}xx^{-1}x \\
      & =x,
    \end{align*}
    as required.
  \end{proof}

  The next lemma provides a relatively efficient means of checking
  whether or not the set $H_{e, f} \cap x/\rho$
  is empty.

  \begin{lem}\label{lem-easy-game}
    Suppose that $x\in S$ is arbitrary and that $e, f\in E(S)$ are
    such that $(e, f)\in \D^S$. If
    $H_{e, f} \cap x/\rho \neq \varnothing$, then $(e, xx ^{-1}), (f,
    x^{-1}x)\in \Tr(\rho)$.
  \end{lem}
  \begin{proof}
    Suppose that $y\in H_{e, f} \cap x/\rho$. Then $(y, e) \in \L^S$
    and $(y, f) \in
    \R ^S$, and so $y^{-1}y = e$ and $yy ^{-1} = f$. Since $(x, y)\in \rho$, it
    follows that $(xx^{-1}, yy^{-1}) = (xx^{-1}, f), (x^{-1}x,
    y^{-1}y) = (x^{-1}x,
    e) \in \Tr(\rho)$, as required.
  \end{proof}

  Clearly, the set $x/\rho$ is the union of the sets $H_{e, f}\cap
  x/\rho$ where $e, f\in E(S)$ and
  $H_{e, f}\cap x/\rho \not=\varnothing$. The contrapositive of
  \cref{lem-easy-game} implies that it suffices to consider those $e,
  f\in E(S)$ such that $(e, xx^{-1}), (f, x^{-1}x)\in Tr(\rho)$.

  The algorithm for iterating through the elements of the set
  $x/\rho$ is then:
  \begin{enumerate}[label=(X\arabic*)]
    \item\label{item-X1} determine the data structure for the semigroup $S$
      consisting of: the generating set $X$, the word graph $\Gamma_X$, the
      strongly connected components of $\Gamma_X$; and one group $\H$-class per
      strongly connected component of $\Gamma_X$ using the algorithms
      described in
      \cite[Section 5.6]{East2019aa};
    \item\label{item-X2} determine the data structure for the quotient $S/\rho$
      consisting of: the generating set $X$; the quotient word graph
      $\Gamma_X/\Tr(\rho)$; the strongly connected components of
      $\Gamma_X/\Tr(\rho)$; and the quotient groups $G/N$ using the algorithms
      described in \cref{section-compute-the-trace}
      and~\cref{section-compute-the-groups};
    \item\label{item-X3} for every pair $\{e, f\}$ of idempotents where
      $(e,xx^{-1}), (f, x^{-1}x) \in\Tr(\rho)$ and $e$ and $f$ belong
      to the same
      strongly connected component of $\Gamma_X$ determine the set $H_{e, f}\cap
      x/\rho$ using \cref{thm-preimages}.
  \end{enumerate}

  To summarise, we compute $H_{e, f} \cap x/\rho$ for every $e, f\in E(S)$
  satisfying the conditions of \ref{item-X3}. These conditions
  suffice because if
  $(e,xx^{-1})\not \in\Tr(\rho)$ or $(f, x^{-1}x) \not\in\Tr(\rho)$, then $H_{e,
  f}\cap x/\rho = \varnothing$ by the contrapositive of \cref{lem-easy-game}. On
  the other hand, $e$ and $f$ belong to the same strongly connected component of
  $\Gamma_X$ if and only if $e\D ^S f$. If $e$ and $f$ are not
  $\D^S$-related, then
  $H_{e, f}$ is empty and so $H_{e, f} \cap x/\rho$ is too. Thus the
  idempotents $\{e, f\}$ satisfying the
  conditions of \ref{item-X3} include all such sets such that $H_{e,
  f}\cap x/\rho \neq \varnothing$.

  It might be worth noting that in the case that $\{e, f\}$ satisfy
  the conditions of \ref{item-X3}, but $H_{e, f} \cap x/\rho =
  \varnothing$, then the set
  $\left(H_{(e)\phi_x, (e)\phi_x}\cap e/\rho\right)xs^{-1}$ (from
  \cref{thm-preimages}) has empty intersection with $\im(\phi_x) \cap
  H_{e, e}$ and so $(\left(H_{(e)\phi_x, (e)\phi_x}\cap
  e/\rho\right)xs^{-1})\phi_x|_{H_{e, e}} ^{-1}$ is empty.

  \begin{ex}
    We continue \cref{ex-norm-subgroups} by computing \(x/\rho\)
    where \(x=[1\ 2\
    4] (3)\in I_4\). Steps \ref{item-X1} and \ref{item-X2} were covered in
    \cref{ex-trace} and \cref{ex-norm-subgroups}, respectively. For step
    \ref{item-X3}, we iterate though all the pairs of idempotents \(e, f\) such
    that \[e=_\rho xx^{-1}=(1)(2)(3) \quad\text{ and }\quad f=_\rho
    x^{-1}x=(2)(3)(4).\] In this case, there is only one such pair when \(e =
    (1)(2)(3)\) and \(f = (2)(3)(4)\). We then compute
    \(\left(\left(H_{(e)\phi_x,
    (e)\phi_x}\cap e/\rho\right)xs^{-1}\right) \phi_x|_{H_{e, e}}^{-1}\cdot s\)
    where \(s\in I_4\) is any fixed element such that \(s^{-1}es=f\); such as
    \(s=[1\ 2\ 3\ 4]\). Since \(\Tr(\rho)\) equals \(\Delta_{D_x\cap
      E(S)}=\set{(d,
    d)}{d\in D_x\cap E(S)}\) when restricted to the idempotents of the
    \(\mathscr{D}\)-class of \(x\), it follows that \(\phi_x\) is the identity
    function. Thus \(\left(\left(H_{e, e}\cap e/\rho\right)xs^{-1}\right) \cdot
    s=\left(H_{e, e}\cap e/\rho\right)x \).  We calculated in the previous
    example that \(H_{e, e}\cap e/\rho\) is the alternating group on \(\{1, 2,
    3\}\), that is \{(1)(2)(3), (1\ 2\ 3), (1\ 3\ 2)\}. Translating this by
    \(x=[1\ 2\ 4] (3)\) gives
    \[
      x/\rho =  \{(1)(2)(3)\cdot[1\ 2\ 4] (3),
        (1\ 2\ 3)\cdot[1\ 2\ 4] (3), (1\ 3\
      2)\cdot[1\ 2\ 4] (3)\}=\{[1\ 2\ 4] (3), [1\ 4] (2\ 3), [1\ 3\ 4] (2)\}.
    \]
  \end{ex}

  \section{Testing membership}\label{section-membership}

  In this section we address how to test whether or not a pair $(a, b)\in S
  \times S$ belongs to the congruence $\rho$. It is well-known that $(a, b)\in
  \rho$ if and only if $(a^{-1}a, b^{-1}b)\in \Tr(\rho)$ and $ab ^ {-1}\in
  \Ker(\rho)$; see, for example, \cite[Theorem 5.3.3]{Howie}.
  \cref{lem-test-mem-trace} shows how to check whether or not $(a^{-1}a,
  b^{-1}b)$ belongs to $\Tr(\rho)$: factorise $a^{-1}a$ and $b^{-1}b$ as words
  $u$ and $v$ in the generators $X$ of $S$, and simply check whether or not the
  paths with source $1_S$ labelled by $u$ and $v$ lead to the same node. We can
  also find the elements of $\Ker(\rho)$ as described at the start of
  \cref{section-compute-the-kernel} and check whether or not $ab ^{-1}$ belongs
  to this set of elements. If $|\Ker(\rho)|$ is relatively small, then this
  approach may be satisfactory. However, in many examples (for
    example those from
  \cref{figure-kernel} in \cref{appendix-perf}), it appears that $|\Ker(\rho)|$
  is sufficiently large that this approach is not sufficiently performant.

  The next theorem establishes an alternative means of testing membership in
  $\Ker(\rho)$ which avoids computing the elements of $\Ker(\rho)$.

  \begin{theorem}\label{theorem-mem-test}
    If $S$ is an inverse semigroup, $x\in S$, and $\rho$ is a congruence on $S$,
    then $x\in \Ker(\rho)$ if and only if \((xx^{-1}, x^{-1}x) \in
    \Tr(\rho)\) and
    $(x)\phi_x \in \left(H_{(xx^{-1})\phi_x, (xx^{-1})\phi_x)}\cap
  xx^{-1}/\rho\right)x^2$.
\end{theorem}

\begin{proof}
  Suppose that $x\in S$ is arbitrary.
  If $(xx^{-1}, x^{-1}x)\in \Tr(\rho)$, then $x$, \(s=xx^{-1}\), and
  \(e=f=xx^{-1}\) satisfy the hypotheses of \cref{thm-preimages}, which then
  states:
  \begin{equation}\label{equation-complicated}
    H_{xx^{-1}, xx^{-1}} \cap x/\rho = \left(\left(H_{(xx^{-1})\phi_x,
    (xx^{-1})\phi_x}\cap xx^{-1}/\rho\right)x^2 x^{-1}\right)
    \phi_x|_{H_{xx^{-1},
    xx^{-1}}}^{-1} xx^{-1}.
  \end{equation}
  Conversely, if $x\in \Ker(\rho)$, then $(xx^{-1}, x^{-1}x)\in
  \Tr(\rho)$ by \ref{item-CP2} and so \eqref{equation-complicated} holds again.
  In particular, when proving either implication in the statement of
  the theorem:
  $(xx^{-1}, x^{-1}x)\in \Tr(\rho)$ and \eqref{equation-complicated} holds.

  If \(x \in \Ker(\rho)\), then \(x / \rho = x^{-1} / \rho\), and so
  \(xx^{-1} / \rho = (x / \rho) (x / \rho) = x / \rho\). Conversely, if
  \((x, xx^{-1}) \in \rho\), then \(x \in \Ker(\rho)\).
  Hence
  \begin{align*}
    x\in \Ker(\rho) &\iff (x, xx^{-1})\in\rho \\
    &\iff xx^{-1}\in x/\rho\\
    &\iff xx^{-1}\in x/\rho\cap H_{xx^{-1}, xx^{-1}}\\
    &\iff    xx^{-1}\in\left(\left(H_{(xx^{-1})\phi_x, (xx^{-1})\phi_x}\cap
    xx^{-1}/\rho\right)x^2x^{-1}\right)
    \phi_x|_{H_{xx^{-1},xx^{-1}}}^{-1} xx^{-1} \\
    &\iff    xx^{-1}\in\left(\left(H_{(xx^{-1})\phi_x, (xx^{-1})\phi_x}\cap
    xx^{-1}/\rho\right)x^2x^{-1}\right)
    \phi_x|_{H_{xx^{-1},xx^{-1}}}^{-1} \\
    &\iff    (xx^{-1})\phi_x\in\left(H_{(xx^{-1})\phi_x, (xx^{-1})\phi_x}\cap
    xx^{-1}/\rho\right)x^2x^{-1}\\
    &\iff    (xx^{-1}x)\phi_x\in\left(H_{(xx^{-1})\phi_x, (xx^{-1})\phi_x}\cap
    xx^{-1}/\rho\right)x^2x^{-1}x & \text{using \cref{lem-phi-1}(i)}\\
    &\iff    (x)\phi_x\in\left(H_{(xx^{-1})\phi_x, (xx^{-1})\phi_x}\cap
    xx^{-1}/\rho\right)x^2.\qedhere
  \end{align*}
\end{proof}

\cref{theorem-mem-test} reduces the problem of testing membership in
$\Ker(\rho)$ to that of checking membership in $\Tr(\rho)$ and in the
coset $\left(H_{(xx^{-1})\phi_x, (xx^{-1})\phi_x}\cap
xx^{-1}/\rho\right)x^2$ of the normal subgroup
$\left(H_{(xx^{-1})\phi_x, (xx^{-1})\phi_x}\cap
xx^{-1}/\rho\right)$ (again the representative of this coset is
$(xx^{-1})\phi_x x^2$ rather than $x^2$).

At this point, we have shown how to compute answers to the common questions
about a congruence of an inverse semigroup or monoid without explicitly
computing the kernel of the congruence.

\section{Meets and joins}\label{section-meets-joins}

In this section we briefly outline how to compute the meet or join of
congruences $\sigma$ and $\rho$ on an inverse semigroup $S$ using the data
structure from~\ref{item-Q1},~\ref{item-Q2},~\ref{item-Q3}, and~\ref{item-Q4}
for $\sigma$ and $\rho$.

Suppose that $S$ is a (not necessarily inverse) semigroup, that $X$ is a
generating set for $S$, and that $\rho$ is a congruence on $S$. Then $\rho$ is
uniquely determined by a word graph with nodes $S/\rho$ and edges $(s/\rho, x,
sx/\rho)$ for every $s/\rho \in S/\rho$ and every $x\in X$; see~\cite[Theorem
3.7 and Corollary 3.8]{Merkouri2023aa} for details. It is shown in
\cite[Section 6]{Merkouri2023aa} that a slightly modified version of the
Hopcroft-Karp Algorithm~\cite{Hopcroft1971aa}, for checking whether or not
two finite state automata recognise the same language, can be used to
determine the word graph of the join $\sigma \vee \rho$ of congruences $\sigma$
and $\rho$ on $S$.

The following lemma is required to prove that the following algorithm
for computing the join of two inverse semigroup congruences is valid.

\begin{lem}\label{lem-join-traces}
  Let \(S\) be an inverse semigroup and let \(\rho\) and \(\sigma\) be
  congruences on $S$. Then
  \(\Tr(\rho \vee \sigma) = \Tr(\rho) \vee \Tr(\sigma)\).
\end{lem}
\begin{proof}
  Certainly,  \(\Tr(\rho \vee \sigma) \supseteq \Tr(\rho) \vee \Tr(\sigma)\).

  For the converse containment, it suffices to show that there exists a
  congruence $\tau\subseteq S \times S$ such that $\Tr(\tau) = \Tr(\rho)
  \vee \Tr(\sigma)$ and $\rho\vee \sigma \subseteq \tau$.
  By \cite[Proposition 5.3.4]{Howie}, if $\upsilon
  \subseteq E(S)\times E(S)$ is any normal congruence on $E(S)$, then
  the maximum
  congruence with trace equal to $\upsilon$ is
  \[
    \upsilon_{\max} = \{(a, b)\in S \times S\mid (a^{-1}ea, b^{-1}eb)\in
    \upsilon \text{ for all }e\in E(S)\}.
  \]
  Suppose that $\upsilon, \upsilon'\subseteq E(S)\times E(S)$ are normal
  congruences. Then it is routine to verify that
  \begin{equation}\label{eq-routine}
    \upsilon\subseteq \upsilon' \Rightarrow \upsilon_{\max}\subseteq
    \upsilon'_{\max}.
  \end{equation}

  Since $\Tr(\rho)$ and
  $\Tr(\sigma)$ are normal congruences, by~\cite[Corollary
  III.2.1]{petrich_book}
  their join $\Tr(\rho)\vee \Tr(\sigma)$ is normal also. Hence we may
  define
  \[
    \tau := \left(\Tr(\rho) \vee \Tr(\sigma)\right)_{\max}.
  \]
  By definition, $\Tr(\tau) = \Tr(\rho)\vee \Tr(\sigma)$. It remains
  to show that
  $\rho\vee \sigma \subseteq \tau$, for which it suffices to show that
  $\rho\subseteq \tau$ and $\sigma\subseteq \tau$. We prove the
  former, the proof
  of the latter is identical.

  Clearly $\Tr(\rho)\subseteq \Tr(\rho) \vee \Tr(\sigma)$, and so, by
  \eqref{eq-routine}, $\Tr(\rho)_{\max}\subseteq \left(\Tr(\rho) \vee
  \Tr(\sigma)\right)_{\max} = \tau$. By definition $\rho \subseteq
  \Tr(\rho)_{\max}$, and so $\rho \subseteq \Tr(\rho)_{\max} \subseteq \tau$, as
  required.
\end{proof}

Given \cref{lem-norm-cong-quot,lem-join-traces}, it is straightforward to
verify that if $\sigma$ and $\rho$ are congruences on an inverse semigroup $S$
generated by $X\subseteq S$ and represented by the data structure
from~\ref{item-Q1},~\ref{item-Q2},~\ref{item-Q3}, and~\ref{item-Q4}, then the
data structure for the join of $\sigma$ and $\rho$ can be obtained as follows:
\begin{enumerate}[label=(J\arabic*)]
  \item\label{item-J1}
    compute the word graph $\Gamma_X/\Tr(\sigma\vee \rho)$ using
    \cite[Algorithm 5]{Merkouri2023aa} (the Hopcroft-Karp
    Algorithm~\cite{Hopcroft1971aa});
  \item\label{item-J2}
    compute the strongly connected components of
    $\Gamma_X/\Tr(\sigma\vee \rho)$;
  \item\label{item-J3}
    compute the generating sets for one group $\H$-class per strongly connected
    component of $\Gamma_X/\Tr(\sigma\vee \rho)$ using~\ref{item-N1}
    and~\ref{item-N2}.
\end{enumerate}

The meet $\sigma \wedge \rho$ of congruences $\sigma$ and $\rho$ on an inverse
semigroup $S$ can be computed in similar way, where \ref{item-J1} is replaced
with the computation of the word graph $\Gamma_X/\Tr(\sigma \wedge \rho)$ using
\cite[Algorithm 6]{Merkouri2023aa}. Algorithm 6 in \cite{Merkouri2023aa} is a
slightly modified version of the standard algorithm from automata theory for
finding an automaton recognising the intersection of two regular languages.

\section{The maximum idempotent-separating congruence}
\label{section-max-idempotent-sep}

In this section we give a method for computing the maximum
idempotent-separating congruence on a finite inverse subsemigroup of a
symmetric inverse monoid. We achieve this using a description of the maximum
idempotent-separating congruence via centralisers. We begin with the definition
of a centraliser.

If $S$ is a semigroup and $A$ is a subset of $S$, then the \defn{centraliser}
of $A$ in $S$ is the set
\[
  C_S(A) = \set{ s\in S }{sa = as \text{ for all } a \in A}.
\]
Let \(\mu\) be the congruence defined by \(a=_\mu b\) if and only if \(a
e a^{-1} = b e b^{-1}\) for all \(e \in E(S)\).

\begin{lem}[cf. Section 5.2 in \cite{Lawson1998}]
  The congruence \(\mu\) is the maximum idempotent-separating
  congruence on \(S\).
\end{lem}

We have that $\Ker(\mu) = C_S(E(S))$ and since $\mu$ is the maximum
idempotent-separating congruence on an inverse semigroup $S$, $\Tr(\mu) =
\Delta_{E(S)}$. For the remainder of this section, we discuss how to compute
$C_S(E(S))$ when $S\leq I_n$.

If $S$ is an inverse semigroup of partial permutations of degree $n$, and
$X\subseteq \{1, \ldots, n\}$, then the \defn{(setwise) stabiliser} of $X$ with
respect to $S$ is
\[
  \Stab_S(X) = \set{g\in S}{(X)g = X}\leq S.
\]

\begin{prop}
  If $S\leq I_n$ is an inverse semigroup, then
  \[
    C(E(S)) = \bigcup_{e\in E(S)} \bigcap_{f\leq e}
    \Stab_{S\cap\operatorname{Sym}(\dom(e))}(\dom(f)),
  \]
  where \(\operatorname{Sym}(\dom(e))\) denotes the group of
  permutations of \(\dom(e)\), and \(f\) is taken to be in \(S\).
\end{prop}
\begin{proof}
  ($\subseteq$) Let \(s\in C(E(S))\) and \(e= ss^{-1}\). As
  \(se=es=s\) and \(\dom(e)=\dom(s)\) it follows that
  \[
    \dom(e)=\dom(s)=\dom(se)=(\dom(e))s^{-1}.
  \]
  Thus \(s^{-1}\) bijectively maps \(\dom(e)=\dom(s)\) to itself. So \(s\) does
  too, and so \(s\in S\cap\operatorname{Sym}(\dom(e))\). Let \(f\leq e\). To
  conclude that \(s\) belongs to the right side of the equality in the
  statement, it suffices to show that \((\dom(f))s=\dom(f)\). By assumption
  \(fs=sf\), so
  \[
    \dom(f)s=\im(fs)=\im(sf)=\im(s)\cap\dom(f)=\dom(e)\cap\dom(f)=\dom(f).
  \]

  ($\supseteq$)
  Let \(s\) be an element of the right hand side of the equality in the
  statement of the proposition. Then there exists \(e\in E(S)\) such that for
  all \(f\leq e\), we have \(s\in
  \Stab_{S\cap\operatorname{Sym}(\dom(e))}(\dom(f))\). In particular, this
  holds when \(f=e\), and so \(s\in S\cap\operatorname{Sym}(\dom(e))\). Let
  \(g\in E(S)\) be arbitrary. We need to show that \(s g=g s\). Since \(s\) is
  an element of a subgroup with identity \(e\), it follows that
  \(ss^{-1}=s^{-1}s=e\). If we define \(f= eg\), then as \(f \leq e\),
  \((\dom(f))s=\dom(f)\). This implies that \((\dom(f))s^{-1}=\dom(f)\) and so
  \[
    g s = g es = f s = s |_{\dom(f)} = s |_{\dom(f)s^{-1}} = sf =
    seg = sg. \qedhere
  \]
\end{proof}

If \(A\subseteq \mathcal{P}(X)\) for some set \(X\), then we say that \(A\) is
a \defn{boolean algebra} (on \(X\)) if \(A\) is closed under taking finite
(possibly empty) unions, and is also closed under taking complements in \(X\).
Each boolean algebra is partially ordered by \(\subseteq\) and contains the
empty set, which is called the \(0\) of the algebra. The complement of \(0\)
(the universal set) is similarly called the \(1\). Note that this is consistent
with standard meet semilattice notation. If \(Y \subseteq X\) we write \(Y^c\)
to denote the complement of \(Y\) in \(X\).  We say that an element of a
boolean algebra is an \defn{atom} if it is a minimal non-zero element. If \(B\)
is a boolean algebra, then we define \(A(B)\) to be the set of atoms of \(B\).
For any finite boolean algebra \(B\), \(B= \set{\cup Y}{Y\subseteq A(B)}\). If
\(S\leq I_n\) is an inverse semigroup, then we define \(B(S)\leq
\mathcal{P}(\{1, 2, \ldots, n\})\) to be the least boolean algebra containing
the set of domains (or equivalently images) of the elements of \(S\), noting
that such a boolean algebra exists as the intersection of two boolean algebras
is always a boolean algebra.

\begin{theorem}\label{prop-centraliser}
  If \(S\leq I_n\) is an inverse semigroup, then
  \[
    C(E(S))=\set{s\in S}{ (b)s=b\text{ for all }b\in A(B(S)) \text{
    such that }b\subseteq \dom(s)}.
  \]
\end{theorem}
\begin{proof}
  \((\subseteq)\) Let \(s\in C(E(S))\).
  We must show that \((b)s = b\) for all \(b\in A(B(S))\) such that
  \(b\subseteq \dom(s)\). Let \(b\in A(B(S))\) be such that
  \(b\subseteq \dom(s)\) and let
  \begin{align*}
    X&=  \set{Y\subseteq \{1, \ldots, n\}}{(Y)s\subseteq Y\text{
    and }(Y^c)s\subseteq Y^c}\\
    X'&=\set{Y\subseteq \{1, \ldots, n\}}{(Y)s\subseteq Y\text{ and
    }(Y)s^{-1}\subseteq Y}.
  \end{align*}
  We show that \(X=X'\). Let \(Y \in X\). Then \((Y)s \subseteq Y\) and
  \((Y^c)s \subseteq Y^c\). So \(s\) moves nothing from \(Y\) to \(Y^c\) and
  nothing from \(Y^c\) to \(Y\), and thus the same must hold for \(s^{-1}\). In
  particular, \((Y)s^{-1} \subseteq Y\) and so \(Y \in X'\) and \(X \subseteq
  X'\). Now suppose \(Y \in X'\). Then \((Y)s \subseteq Y\) and \((Y)s^{-1}
  \subseteq Y\). The later implies that \(s\) cannot move anything from \(Y^c\)
  into \(Y\) and so \((Y^c)s \subseteq Y^c\) and \(Y \in X\). Thus \(X'
  \subseteq X\) and so \(X' = X\).

  Note that \(X\) is a boolean algebra, as from the definition of \(X\) it is
  closed under complements, and from the definition of \(X'\) it is closed
  under unions. Let \(D = \set{\dom(t)}{t\in S}\). By definition, the least
  boolean algebra containing \(D\) is \(B(S)\). We will show that \(D\subseteq
  X\). This will be sufficient because, together with the fact that \(X\) is a
  boolean algebra, this implies that \(B(S)\subseteq X\), which in turn implies
  that \((b)s\subseteq b\). Since \(b\) is an atom \((b)s\) cannot be a proper
  subset of \(b\).

  So let \(d\in D\) be arbitrary, and let \(f_d\in S\) be an idempotent with
  \(\dom(f_d)=d\). As \(s\in C(E(S))\), we have \(sf_d=f_ds\). The image of
  \(sf_d\) is \(d\cap \im(s)\) and the image of \(f_ds\) is \((d)s\). Thus
  \((d)s=d\cap \im(s)\) and so \((d)s\subseteq d\). Since \(s^{-1}\in
  C(E(S))\), we similarly get that \((d)s^{-1}\subseteq d\). It follows that
  \(d\in X\), as required.
  \medskip

  \((\supseteq)\) Let \(s\in S\) be such that for all \(b\in A(B(S))\) with
  \(b\subseteq \dom(s)\), we have \((b)s=b\). Let \(e\in S\) be an idempotent.
  We will show that \(se=es\). Let \(b_1, \ldots, b_k\in A(B(S))\) be distinct
  such that \(\dom(e)\cap \dom(s)= b_1\cup\ldots\cup b_k\). Note that, from the
  assumption on \(s\), \((b_i)s=b_i\) for all \(1\leq i\leq k\) so
  \[\dom(e)\cap \dom(s)= b_1\cup\ldots\cup b_k=\dom(e)\cap \im(s).\]

  For all \(x\in \{1, \ldots, n\}\) we have that
  \begin{align*}
    (\{x\})es & =\left\{
      \begin{array}{lr}
        \varnothing & \text{ if } x\notin b_1\cup\ldots\cup b_k\\
        \{(x)s\}& \text{ if } x\in b_1\cup\ldots\cup b_k\\
      \end{array}\right. \\
      & =(\{x\})se.
    \end{align*}
    Therefore, \(se=es\) and so \(s\in C(E(S))\), as required.
  \end{proof}

  \begin{ex}
    We compute the maximum idempotent-separating congruence \(\mu\) of the
    semigroup \(I_4\). The first step is to construct \(C(E(I_4))\) using
    \cref{prop-centraliser}. The set of domains of elements of \(I_4\) is just
    \(\mathcal{P}(I_4)\), and so \(B(I_4) = \mathcal P(\{1, 2, 3, 4\})\). It
    follows that \(A(B(I_4))\) is the set of singleton subsets of \(\{1, 2, 3,
    4\}\). From \cref{prop-centraliser}, it follows that
    \[
      C(E(I_4)) = \set{s \in I_4}{(i)s = i \text{ for all } i \in
      \{1, 2, 3, 4\} \text{ such that } i \in \dom(s)}.
    \]
    This implies that \(C(E(I_4))\) is precisely the set of elements of \(I_4\)
    which act as the identity on their domains, which is just \(E(I_4)\) and so
    \(\Ker(\mu) = C(E(I_4)) = E(I_4)\). Since \(\mu\) is idempotent-separating,
    we already know that \(\Tr(\mu) = \Delta_{E(S)}\), and so we have computed
    the kernel and trace for \(\mu\), which fully describes the congruence. In
    this case, the kernel and trace equal those of the trivial congruence
    \(\Delta_{S}\), and so \(\mu = \Delta_{S}\).
  \end{ex}

  \section*{Acknowledgements}
  The authors were supported by a Heilbronn Institute for Mathematical Research
  Small Grant during part of this work. The second named author was supported
  the Heilbronn Institute for Mathematical Research during this work. The
  authors would also like to thank the University of Manchester for hosting
  them during part of the work on this paper; and Reinis Cirpons for
  his assistance in producing the figures in \cref{appendix-perf}.

  \printbibliography

  \appendix

  \section{Performance of the algorithms}\label{appendix-perf}

  In this appendix we provide some empirical evidence for our earlier claims
  about the performance of the algorithms described in this paper; see
  \cref{figure-trace,figure-kernel,figure-nr-classes,figure-max-idem-sep}.
  Each point in these figures represents the mean of
  a number of trials of the relevant computation related to a congruence on an
  inverse semigroup consisting of partial permutations. The number of trials was
  chosen according to the run-time of each computation, with shorter run-times
  having a larger number of trials. Each time is the mean of between 5 runs and
  10,000 runs. The inverse semigroups were chosen at random with between $1$ and
  $n$ generators of degree $n$ for $n\in\{5, 6, \ldots, 9\}$. The congruences
  were given by between $1$ and $5$ generating pairs consisting of randomly
  chosen elements of the corresponding inverse semigroup. Although other samples
  might exhibit different behaviours, and the sample used here is not unbiased,
  the authors believe they do provide some indication of the relevant merits of
  the algorithms presented in this article.

  \cref{figure-trace} contains a comparison of a preliminary implementation of
  the algorithm from \cref{section-compute-the-trace} with the earlier
  implementation in \cite{Semigroups} described in \cite{Torpey}. It should be
  noted that the algorithm described in \cref{section-compute-the-trace} permits
  the computation of the trace of a congruence $\rho$ on an inverse semigroup
  without any computation of the kernel of $\rho$. The algorithm described
  in~\cite{Torpey} and implemented in~\cite{Semigroups} computes the
  trace at the
  same time as computing the kernel. Estimating the complexity from
  \cref{figure-trace} the existing algorithm from~\cite{Semigroups,Torpey} has
  complexity approximately $O(|S|)$ where the algorithm based on
  \cref{section-compute-the-trace} has complexity $O(|S|^{0.31})$.

  \cref{figure-kernel} contains a comparison of a preliminary implementation of
  an algorithm (based on \cref{section-compute-the-kernel}) for computing the
  kernel of a congruence $\rho$ on an inverse semigroup $S$. This algorithm is
  rather simplistic, it computes representatives $e_1, \ldots, e_k\in
  S$ of trace
  classes (from the word graph $\Gamma_X/\Tr(\rho)$
  from~\ref{item-T2}), and then
  applies~\ref{item-X3} and \cref{thm-preimages} to compute the elements of the
  class $e_i/\rho$ for every $i$. Estimating the complexity from
  \cref{figure-kernel} the existing algorithm from~\cite{Semigroups,Torpey} has
  complexity approximately $O(|S|)$ where the algorithm based on
  \cref{section-compute-the-kernel} has complexity $O(|S|^{0.49})$. We reiterate
  the point (made several times earlier in this article) that computing the
  kernel is not required to answer most questions about congruences on inverse
  semigroups, although it might be interesting in its own right.

  \cref{figure-nr-classes} contains a comparison of the run-times of the
  following for computing the number of classes of a congruence $\rho$ on an
  inverse semigroup $S$: an implementation of the algorithm
  from~\cref{section-compute-the-trace} (specifically
  \eqref{equation-nr-classes});
  the earlier implementation in \cite{Semigroups} described in
  \cite{Torpey}; and
  the implementation in \cite{libsemigroups} and \cite{Semigroups} for computing
  a congruence on a (not necessarily inverse) semigroup. The latter
  does not make
  use of the fact that the input semigroups are inverse.
  Estimating the complexity from \cref{figure-nr-classes}
  the existing algorithm from~\cite{Semigroups,Torpey} has complexity
  approximately $O(|S|^{1.05})$; the algorithm based on
  \cref{section-compute-the-kernel} has complexity approximately
  $O(|S|^{0.25})$;
  and the generic method from \cite{libsemigroups, Semigroups} has complexity
  approximately $O(|S|^{1.34})$.

  \cref{figure-max-idem-sep} contains run-times of an implementation of the
  algorithm described in~\cref{section-max-idempotent-sep}. The inverse
  semigroups used to produce \cref{figure-max-idem-sep} were generated as
  described above. The authors of this paper are not aware of any existing
  algorithms in the literature for computing the maximum idempotent separating
  congruence of an inverse semigroup, and as such there is no comparison in
  \cref{figure-max-idem-sep}. Estimating the complexity from
  \cref{figure-max-idem-sep} the algorithm based on
  \cref{section-max-idempotent-sep} has complexity $O(|S|^{0.48})$.

  \begin{figure}
    \includegraphics[width=0.7\textwidth]{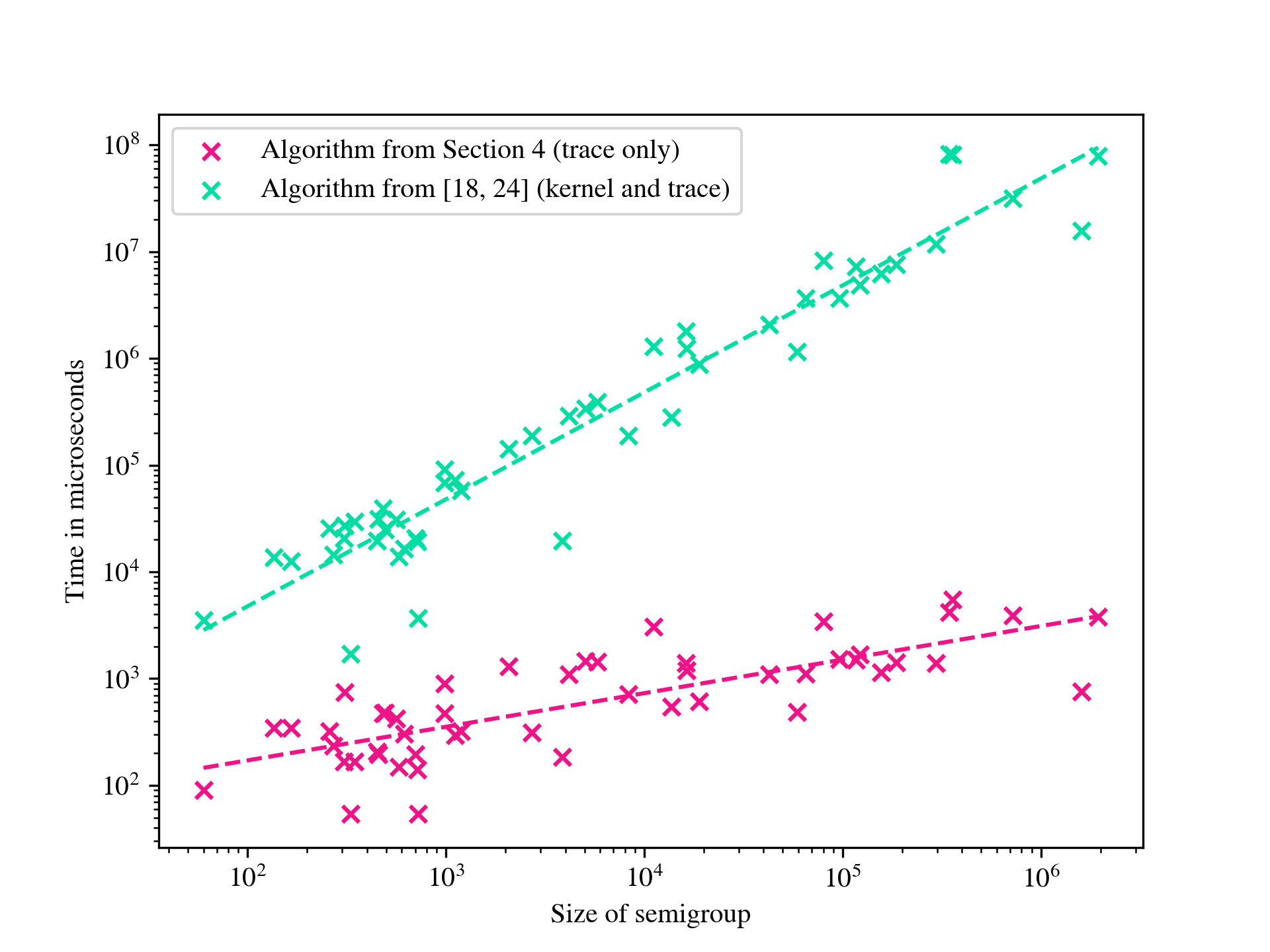}
    \caption{Comparison of the run-times of an implementation of the
      algorithm described in \cref{section-compute-the-trace} and the earlier
      implementation in~\cite{Semigroups} described in \cite{Torpey}
    for computing the trace of a congruence on an inverse semigroup.}
    \label{figure-trace}
  \end{figure}

  \begin{figure}
    \includegraphics[width=0.7\textwidth]{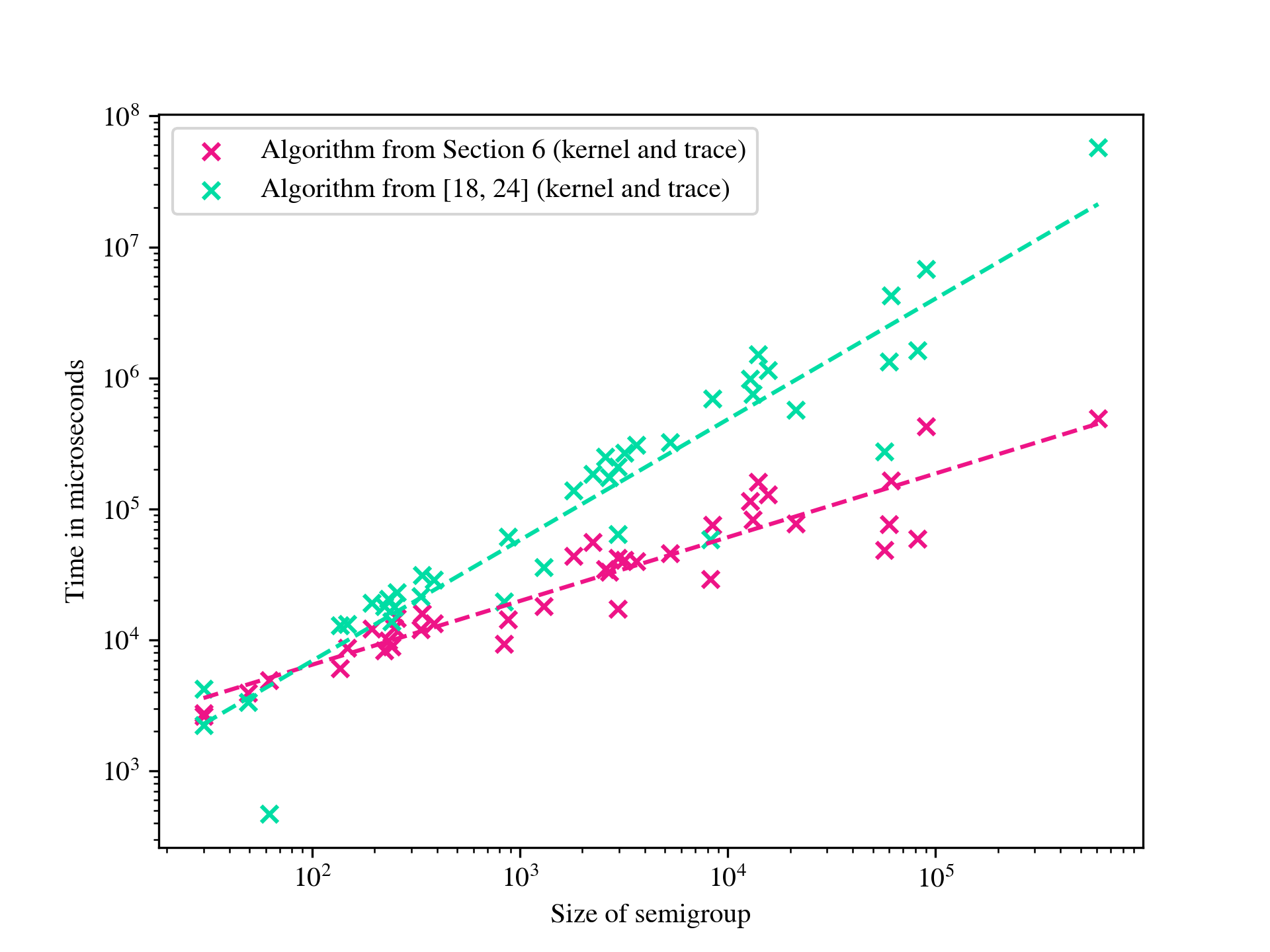}
    \caption{Comparison of the run-times of an implementation of the
      algorithm described in \cref{section-compute-the-kernel} and the earlier
      implementation in~\cite{Semigroups} described in \cite{Torpey}
    for computing the kernel and trace of a congruence on an inverse semigroup.}
    \label{figure-kernel}
  \end{figure}

  \begin{figure}
    \includegraphics[width=0.7\textwidth]{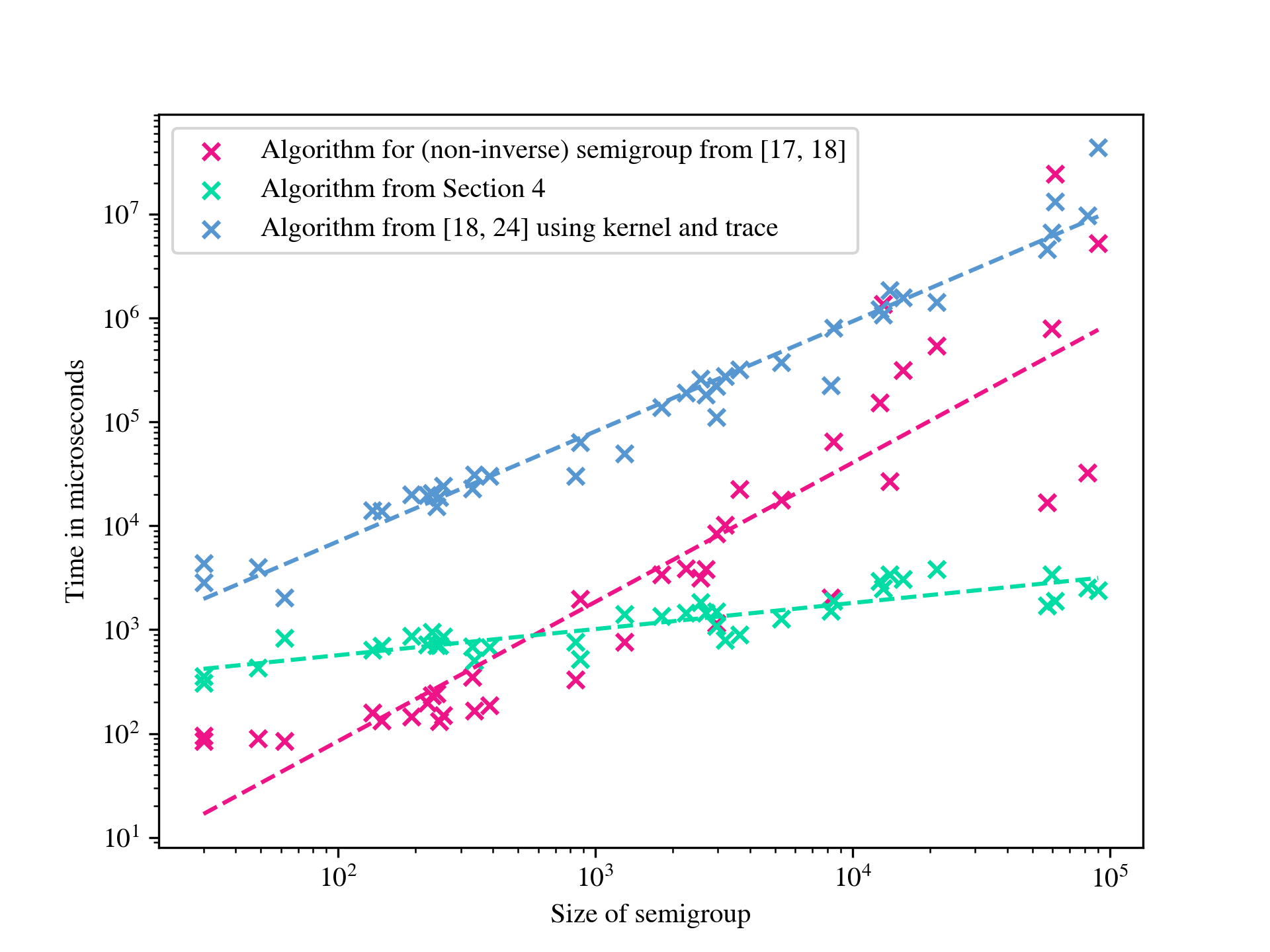}
    \caption{Comparison of the run-times of the following algorithms
      for computing
      the number of classes of a congruence: the implementation of the algorithm
      described in \cref{section-compute-the-trace}; the earlier implementation
      in~\cite{Semigroups} described in \cite{Torpey} using the
      kernel and trace;
      and the algorithm implemented in~\cite{libsemigroups}
      and~\cite{Semigroups} for finding a congruence on a (not
      necessarily inverse)
    semigroup.}
    \label{figure-nr-classes}
  \end{figure}

  \begin{figure}
    \includegraphics[width=0.7\textwidth]{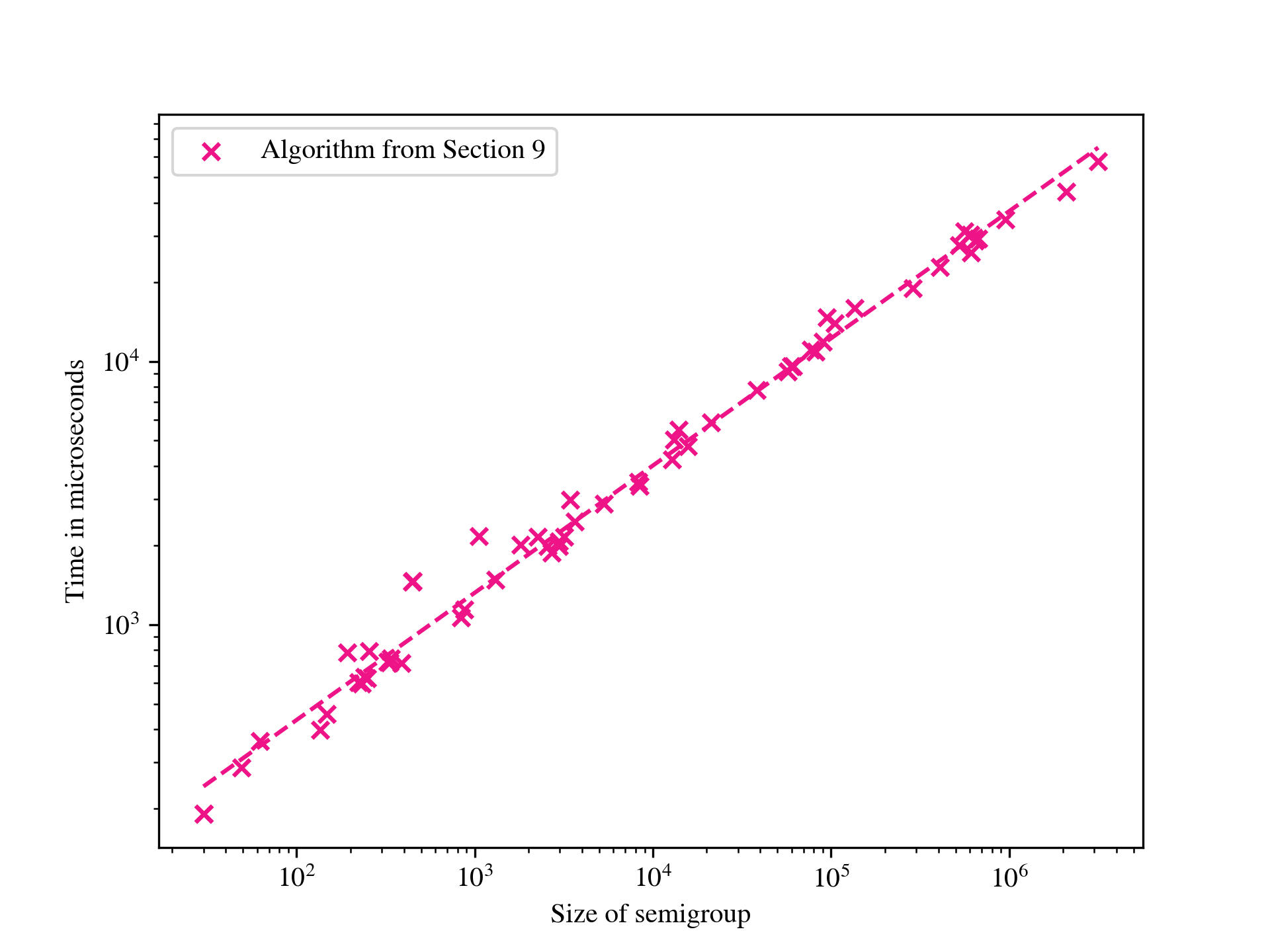}
    \caption{The run-times of an implementation of the algorithm described in
      \cref{section-max-idempotent-sep} for computing the maximum idempotent
    separating congruence of an inverse semigroup.}
    \label{figure-max-idem-sep}
  \end{figure}

  \end{document}